\documentclass{amsart}

\usepackage[T1]{fontenc}
\usepackage{amsmath,amsthm}
\usepackage{amssymb, enumitem, mathrsfs, verbatim, bm, esint}
\usepackage{amstext}
\usepackage[top=3cm,bottom=3cm,left=3cm,right=3cm]{geometry}

\usepackage{hyperref}

\newtheorem{theorem}{\bf Theorem}[section]

\newtheorem{lemma}{\bf Lemma}[section]
\newtheorem{definition}{\bf Definition}[section]

\newcommand{\N}{\mathbb N}

\newcommand{\R}{{\mathbb R}}
\newcommand{\ep}{\varepsilon}

\numberwithin{equation}{section}
\newcommand{\diver}{\operatorname{div}}

\newcommand{\dert}{\partial_t}

\newcommand{\q}{{\bf q}}
\newcommand{\vc}[1]{{\bf #1}}
\newcommand{\vv}{\vc{v}}
\newcommand{\ww}{\mathbf{w}}

\newcommand{\thet}{\vartheta}

\renewcommand{\S}{\mathbf{S}}
\newcommand{\D}{\mathbf{D}}

\newcommand{\dt}{\,{\rm d} t }
\newcommand{\ds}{\,{\rm d} s }
\newcommand{\dx}{\,{\rm d} {x}}
\newcommand{\dxdt}{\dx  \dt}

\newcommand{\bphi}{\boldsymbol{\varphi}}

\newcommand{\intO}[1]{\int_{\Omega} #1 \ \dx}
\newcommand{\intTO}[1]{\int_0^T \int_{\Omega} #1 \ \dxdt}

\newcommand{\f}{{\bf f}}

\newcommand{\T}{\mathcal{T}}
\newcommand{\Tk}{\T_k}

\newcommand{\gk}{g_k}
\newcommand{\n}{\mathbf{n}}

\begin{document}

\title[non-Newtonian Navier-Stokes-Fourier system with dissipative heating]{On the existence of solutions to generalized Navier--Stokes--Fourier system with dissipative heating}
\author[A.~Abbatiello]{Anna Abbatiello}
\address{Anna Abbatiello.
University of Campania ``Luigi Vanvitelli'', Department of Mathematics and Physics, viale A.~Lincoln 5, 81100 Caserta, Italy.}   
\email{\tt anna.abbatiello@unicampania.it}

\author[M.~Bul\'{i}\v{c}ek]{Miroslav Bul\'i\v{c}ek}
\address{Miroslav Bul\'i\v{c}ek. Mathematical Institute, Faculty of Mathematics and Physics, Charles University, Sokolovsk\'{a} 83, 18675 Prague, Czech Republic.}
\email{\tt mbul8060@karlin.mff.cuni.cz}

\author[D.~Lear]{Daniel Lear}
\address{Daniel Lear.  Universidad de Cantabria, Departamento de Matem\'{a}ticas, Estad\'{i}stica y Computaci\'{o}n, Avd. Los Castros s/n, 39005 Santander, Spain. }
 \email{\tt daniel.lear@unican.es}

%%%% Subject entries to be placed here %%%%

%%%% Keyword entries to be placed here %%%%
\keywords{Navier--Stokes--Fourier equations, dissipative heating, Oberbeck--Boussinesq system,  non-Newtonian fluids, large-data existence}
\subjclass[2020]{35Q30, 35K61, 76E30}
\thanks{The research activity of A.~Abbatiello is performed under the auspices of GNFM - INdAM. M. Bul\'{\i}\v{c}ek and D. Lear acknowledge the support of the project  No. 20-11027X financed by Czech Science Foundation (GA\v{C}R). M. Bul\'{\i}\v{c}ek is a member of the~Ne\v{c}as Center for Mathematical Modelling. D. Lear was partially supported by the RYC2021-030970-I research grant, and the AEI grants PID2020-114703GB-I00 and PID2022-141187NB-I00.\\
Corresponding author email: {\tt anna.abbatiello@unicampania.it}}

\begin{abstract}
We consider a flow of non-Newtonian incompressible heat conducting fluids with dissipative heating. Such system can be obtained by scaling the classical Navier--Stokes--Fourier problem. As one possible singular limit may be obtained the so-called Oberbeck--Boussinesq system. However, this model is not suitable for studying the systems with high temperature gradient. These systems are described in much better way by completing the Oberbeck--Boussinesq system by an additional dissipative heating. The satisfactory existence result for such system was however not available. In this paper we show the large-data and the long-time existence of dissipative and suitable weak solution. This is the starting point for further analysis of the stability properties of such problems.
\end{abstract}

\maketitle

\section{Introduction}\label{S1}

The Rayleigh--B\'{e}nard problem of thermal convection is one of the most canonical examples of instability in the fluid flow and it has attracted lot of attention not only in the modelling and physical understanding of such phenomenon  but also in the mathematical community. Indeed, this became a source of difficult mathematical problems and has driven the study of the so-called singular limits during last decades. Such singular limits may have many forms depending on the scaling, and the most classical model (the asymptotic limit) is the so-called Oberbeck--Boussinesq system, that is used exactly for Rayleigh--B\'{e}nard convection. This is the case when the fluid is heated below with the bottom temperature $\theta^{b}$ and with prescribed lower temperature $\theta^{t}$ on the top of two parallel plates and when the fluid can be understood as mechanically incompressible and all changes in the density are just because of variations in the temperature. It appears that such model is the very accurate approximation of real fluid in case that the temperature gradient  and consequently $(\theta^b - \theta^t$) is not large. On the other hand, for large temperature gradient, the Oberbeck--Boussinesq may significantly fail in giving the correct predictions and therefore there are many attempts how to generalize the approximative model. One of such attempt is to add a dissipative heating into the system. This approach was successfully used in \cite{RRS} and \cite{KRT}, where the authors formally derived a generalization of the Oberbeck--Boussinesq system in the following form
\begin{subequations}\label{sys:fullmodel}
\begin{align*}
\diver \vv &= 0,\\
\dert \vv + \diver(\vv \otimes \vv ) - \frac{1}{\sqrt{\text{Gr}}}\diver \S  +\nabla p&= - \theta\mathbf{f}, \\
\tilde{c}_v(\theta)\left(\dert\theta+\diver(\theta\vv)\right) + \frac{1}{\text{Pr}\sqrt{\text{Gr}}}\diver \vc{q} &= \frac{2\text{Di}}{\sqrt{\text{Gr}}}\S:\D{\vv} +\text{Di}(\theta+\Theta) (\vv\cdot\mathbf{f}).
\end{align*}
\end{subequations}
Here, $\text{Gr}$ is the  Grashof number, $\text{Pr}$ denotes the Prandtl number, and the new additional scaling parameter is the dissipative number  $\text{Di}$. Further, $\vv$ means the velocity field, the body force is just the gravity, i.e. it has direction in the $d$-th component and is constant $\mathbf{f}:=(-1)\mathbf{e}_d$ and the auxiliary functions $\Theta$ scales the difference between the temperature on the top and on the bottom of plates
\[
\Theta:=\frac{1}{2}\frac{\theta^{b}+\theta^{t}}{\theta^{b}-\theta^{t}}.
\]
Here $\theta$ is the real physical temperature. Note here that assuming $\text{Di}=0$, we arrive to the classical Oberbeck--Boussinesq system without dissipative heating. This system with $\text{Di}=0$ is rigorously derived and treated in \cite{BeFeOs23} as the singular limit of compressible Navier--Stokes--Fourier system with certain nonlocal boundary conditions. The existence analysis for such problem is then established in~\cite{AbFe23}. It should be mentioned here that all considered models in \cite{BeFeOs23,AbFe23} are Newtonian, i.e. the Cauchy stress $\S$ is linear with respect velocity gradient.

\subsection{Beyond Oberbeck--Boussinesq system}
Here, in this paper, we want to deal with $\text{Di}>0$ and with possibly nonlinear $\S$, and as a starting point for further analysis, we want to establish the existence of a solution. To simplify the notation and also computation (but without any essential impact on the analysis), we denote a new unknown $\thet:=\theta+\Theta$, consider general body force $\mathbf{f}$ and scale the equations such that $\text{Gr}=\text{Pr}=\text{Di}=1$ and we obtain the final system
%\mathbf{f}=(-1)\mathbf{e}_d\qquad \text{Di}=\text{Gr}=...=1
%\]
%After all these simplifications we arrive to the ``extended'' Oberbeck-Boussinesq system
%\begin{align*}
%\diver \vv &= 0\\
%\dert \vv + \diver(\vv \otimes \vv ) - \diver \S + \nabla \pi &= \thet\mathbf{e}_d \\
%\dert\thet+\diver(\thet\vv) + \diver \vc{q} &= \S:\D{\vv} -\thet (\vv\cdot\mathbf{e}_d)
%\end{align*}
%
%\cite{BeFeOs23}
%
%\newpage
%\section{The Oberbeck--Boussinesq system}
%We are interested in the following system
\begin{subequations}\label{al}
\begin{align}
\label{i1}
\diver \vv &= 0,\\
\label{i2}
\dert \vv + \diver(\vv \otimes \vv ) - \diver \S + \nabla \pi &= \thet\mathbf{f}, \\
\label{i3}
\dert\thet+\diver(\thet\vv) + \diver \vc{q} &= \S:\D -\thet (\vv\cdot\mathbf{f}),
\end{align}
\end{subequations}
which is supposed to be satisfied in $Q:=(0,T) \times \Omega\subset(0, +\infty)\times \R^d$ with $\Omega$  being a Lipschitz domain. Note here, that \eqref{i1} is the incompressibility constraint, \eqref{i2} is the balance of linear momentum and \eqref{i3} is the balance of internal energy. Here $\vv: Q\to \R^d$ denotes the velocity field,  $\D:=  (\nabla\vv + (\nabla\vv)^{T})/2$ is the symmetric part of the velocity gradient $\nabla \vv$, $\pi:Q\to \R$ is the pressure, $\thet:Q\to \R$ is the temperature; $\S: Q\to \R^{d\times d}_{\rm sym}$ denotes the viscous part of the Cauchy stress tensor and $\vc{q}:Q \to \R^d$ is the heat flux.

\subsection{Cauchy stress tensor and heat flux}

The heat flux $\vc{q}$ is represented by the Fourier law
\begin{equation}\label{Fourier}
\vc{q} = \vc{q}^*(\thet,\nabla \thet):= - \kappa(\thet) \nabla \thet,
\end{equation}
with the heat conductivity $\kappa: \R \to (0, +\infty)$ being a continuous function of the temperature satisfying, for all $\thet \in (0, +\infty)$ and for some $0<\underline{\kappa}\le \overline{\kappa} <+\infty,$
\begin{align}\label{k}
0<\underline{\kappa}\leq \kappa(\thet) \leq \overline{\kappa}<+\infty.
\end{align}
We also assume that the Cauchy stress is given as $\S=\S^*(\thet, \D\vv)$, where $\S^*: (0,\infty)\times \R^{d\times d}_{\rm sym} \to \R^{d\times d}_{\rm sym}$ is a continuous mapping fulfilling for some $p>2d/(d+2)$, some $0< \underline{\nu}, \overline{\nu}<+\infty$ and for all $\thet\in \R_+$, $\D_1,\D_2 \in \R^{d\times d}_{\rm sym}$
\begin{subequations}\label{nu}
\begin{align}
(\S^*(\thet, \D_1)-\S^*(\thet, \D_2)): (\D_1-\D_2)&\ge 0,\label{nu1}\\
\S^*(\thet,\D_1):\D_1\ge \underline{\nu}|\D_1|^p - \overline{\nu}, \quad |\S^*(\thet, \D_1)|&\le \overline{\nu}(1+|\D_1|)^{p-1}, \quad \S^*(\thet, 0)=0.\label{nu2}
\end{align}
\end{subequations}
The prototypic relation $\S\sim \nu(\thet) |\D\vv|^{p-2}\D\vv$ falls into the class \eqref{nu}.

\subsection{Boundary and initial conditions}\
The system \eqref{i1}--\eqref{i3} is completed by the following initial conditions
\begin{equation}\label{condizioni-iniziali}
\vv(0) = \vv_0, \ \
\thet(0) = \thet_0> 0 \qquad \mbox{ in } \Omega.
\end{equation}
We prescribe the following Navier's slip ($\alpha\geq 0$) boundary condition for the velocity
\begin{equation}\label{BCvelocity}
\vv\cdot\mathbf{n}=0, \qquad {\alpha}\vv_{\tau}=-[\S\mathbf{n}]_{\tau} \qquad \text{on } \partial\Omega,
\end{equation}
and the Neumann type boundary condition for the temperature
\begin{equation}\label{BCtemperature}
\q \cdot \mathbf{n}=0, \qquad \text{on } \partial\Omega.
\end{equation}
Here, $\mathbf{n}$ is the unit outward normal and $\vv_{\tau}$ stands for the projection of the velocity field to the tangent plane, i.e. $\vv_{\tau} := \vv-(\vv\cdot \mathbf{n})\mathbf{n}$. The first condition in \eqref{BCvelocity} expresses the fact that the solid boundary is impermeable, the second condition in \eqref{BCvelocity} is Navier's slip boundary condition, and the condition in~\eqref{BCtemperature} states that there is no heat flux across the boundary. In fact, this is not the correct boundary condition and one should consider here either the Dirichlet boundary condition or inspired by \cite{BeFeOs23} some version of nonlocal boundary condition. However, since we want to present the first large-data result for the problems with dissipative heating, we assume the simplest boundary conditions. However, we are sure that the similar result can be obtained for more realistic boundary conditions and it will be a part of the forthcoming paper about the stability, where the Dirichlet boundary condition plays the crucial role.

%{\remark
%Note that the limit $\alpha\to +\infty$ on the second condition in \eqref{BCvelocity} formally leads to the no-slip boundary condition. This case is however omitted from our studies in what follows. The reason is \myr{explain issues with the pressure...

%\noindent
%Navier-slip BC $\Longrightarrow$ good pressure\\
%Dirichlet BC $\Longrightarrow$ pressure is something nice plus a distribution\\
%Space-periodic BC $\Longrightarrow$ Feireisl and Malek}}

\section{Definition of solution and  main theorem}
We assume that the initial data $\vv_0, \thet_0$ and the external body force $\mathbf{f}$ satisfy
\begin{subequations}\label{propertiesdata}\begin{align}
&\vv_0\in L^2_{\mathbf{n},\diver},\quad  \thet_0\in L^1(\Omega), \quad \log \thet_0\in L^1(\Omega), \quad \mathbf{f}\in L^{\infty}((0,T)\times \Omega),\label{DATA1}
\end{align}
and that
\begin{align}
&\thet_0\geq 0 \mbox{ for a.a. } x\in \Omega.\label{DATA2}
\end{align}\end{subequations}
We look for $(\vv,\thet,\pi):[0,T]\times \Omega \to \R^d\times \R^{+}\times \R$ solving the following set of equations in $(0,T)\times \Omega$
\begin{align}
\label{eq:v}
\dert \vv + \diver(\vv \otimes \vv ) - \diver \S + \nabla \pi = \thet\mathbf{f},\qquad \diver \vv = 0, \\
\label{eq:E}
\dert \left(\frac{|\vv|^2}{2}+\thet\right)+\diver\left(\vv\left(\frac{|\vv|^2}{2}+\thet+\pi\right) \right) + \diver \vc{q} = \diver (\S \vv),
\end{align}
completed by a weak formulation of the internal energy inequality
\begin{equation}\label{ineq:thet}
\dert\thet+\diver(\thet\vv) + \diver \vc{q} \geq  \S:\D -\thet (\vv\cdot\mathbf{f}),
\end{equation}
and satisfying the boundary conditions
\begin{subequations}\label{bcandt=0}
\begin{equation}
\vv\cdot\mathbf{n}=0, \quad \alpha \vv_{\tau}+[\S\mathbf{n}]_{\tau}=0, \quad \nabla\thet\cdot\mathbf{n}=0 \qquad \text{on } (0,T)\times \partial\Omega,
\end{equation}
and the initial conditions
\begin{equation}
\vv(0, \cdot)=\vv_0, \quad \thet(0, \cdot)=\thet_0.
\end{equation}
\end{subequations}
The inequality~\eqref{ineq:thet} can be also replaced by (this is usually used in the setting of compressible fluids, where the a~priori estimates are not sufficient to define \eqref{ineq:thet} in sense of distributions)
\begin{equation}\label{ineq:entr}
\dert\eta+\diver(\eta\vv) + \diver \left(\frac{\vc{q}}{\thet}\right) \geq  \frac{\S:\D}{\thet}-\frac{\vc{q}\cdot \nabla \thet}{\thet^2} - \vv\cdot\mathbf{f},
\end{equation}
where the entropy $\eta$ is defined as $\eta:=\log \thet$.

Below we give a precise formulation of the notion of weak solution but before that we
introduce some notation that will be needed in what follows.

\subsection{Basic definitions and function spaces}
Let $\Omega\subset \R^d$  be a bounded domain with Lipschitz boundary $\partial \Omega$, i.e. $\Omega \in\mathcal{C}^{0,1}$. We say $\Omega \in  \mathcal{C}^{1,1}$ if  the mappings that
locally describe the boundary $\partial\Omega$ belong to $\mathcal{C}^{1,1}$.

We consider the standard Lebesgue, Sobolev and Bochner spaces endowed with the classical norms.
For our purposes, we introduce for arbitrary $q\in[1,\infty)$ the subspaces of vector-valued Sobolev functions given by
\[
W_{\mathbf{n}}^{1,q}:=\overline{\{\vv\in \mathcal{C}^{\infty}(\Omega;\R^d)\cap \mathcal{C}(\overline{\Omega};\R^d) : \text{tr}(\vv)\cdot\mathbf{n} =0 \text{ on } \partial\Omega \}}^{\|\cdot \|_{1,q}}, \qquad W_{\mathbf{n}}^{-1,q'}:=\left(W_{\mathbf{n}}^{1,q}\right)^{\ast},
\]
and
\[
W_{\mathbf{n},\diver}^{1,q}:=\{\vv\in W_{\mathbf{n}}^{1,q} : \diver(\vv) =0 \}, \qquad W_{\mathbf{n},\diver}^{-1,q'}:=\left(W_{\mathbf{n},\diver}^{1,q}\right)^{\ast},
\]
\[
L_{\mathbf{n},\diver}^q:=\overline{\{\vv\in W_{\mathbf{n},\diver}^{1,q}\}}^{\|\cdot \|_{q}}.
\]
Similarly, we consider the classical Sobolev space $W^{1,q}(\Omega)$ and use the standard abbreviation for its dual space $W^{-1,q'}:=(W^{1,q}(\Omega))^*$. Also, in what follows whenever there is $v\in X^*$, $u\in X$, the symbol $\langle v,u\rangle$ means the duality paring in $X$. In case there was possible ambiguity, we would write $\langle v,u\rangle_X$. Notice that all above mentioned space are Banach spaces that are in addition separable provided that $p,q <\infty$. In addition, they are reflexive whenever $p,q\in (1,\infty)$.
%%\myr{Introduce more notation, explain notation $(\cdot,\cdot)$ vs $\langle\cdot,\cdot\rangle$}
%In what follows we focus on the dimension $3$ but the results can be easily accomodated when $d=2$.
%
%
%\bigskip
%
%{\color{red}{\tt TODO: Move}
%
%\subsection{Perfect-slip case}
%In the manuscript we are interested in Navier's slip ($\alpha\geq 0$) boundary conditions for the velocity field.  According to that, we have to consider two different forms of Korn's inequality depending on whether the type of considered boundary conditions leads to the control of the trace of velocity on the boundary or not (see \cite[Section 3.1.2]{BMR}).
Further, we introduce few inequalities used and needed in the text. Since we deal only with the symmetric gradient, we need some form of the~Korn inequality. Since we want to deal with general boundary conditions, we use the following form (see \cite[Lemma 1.11]{BuMR} or \cite[Theorem 11]{BP})
%For the case $\alpha>0$,  it follows from  \cite[Lemma 1.11]{BuMR} that for all $1<p<\infty$ there exist $C_P,C_K > 0$ depending on $\Omega$ and $p$ such that
%\begin{equation*}
%\|\vv\|_{1,p}\leq C_P \|\nabla\vv\|_p\leq C_K (\|\D(\vv)\|_p+\|\vv\|_{2,\partial\Omega}) \qquad \text{for all } \vv\in W_{\mathbf{n}}^{1,p} \text{ with } \vv_\tau\in L^2(\partial\Omega).
%\end{equation*}
%
%For the sake of simplicity and to avoid some technicalities due to the boundary terms, we focus in the Appendix only on the \textit{perfect-slip} case, i.e. $\alpha=0$. This particular situation  requires us to rule out domains that admit nontrivial rigid motions.  From \cite[Theorem 11]{BP} and Poincare's inequality it follows that if $\Omega$ is not axisymmetric and $1 < p < \infty$ there exist $C_P,C_K > 0$ depending on $\Omega$ and $p$ such that
\begin{equation}\label{eq:Korna0}
\|\vv\|_{1,p}\leq  C(p) (\|\D(\vv)\|_p+\|\vv\|_2) \qquad \text{for all } \vv\in W_{\mathbf{n}}^{1,p},
\end{equation}
which is valid for all $p\in (1,\infty)$ provided that $\Omega$ is Lipschitz. Further, we also frequently use in the paper the following interpolation inequality
\begin{equation}\label{interp}
\|u\|_{\frac{p(d+2)}{d}}^{\frac{p(d+2)}{d}}\le C(\Omega,p) \|u\|_{2}^{\frac{2p}{d}} \|u\|_{1,p}^{p}.
\end{equation}

%}

\subsection{Definition of weak and suitable weak solutions and main theorem}
%For all $\psi\in\mathcal{C}(\overline{\Omega})$, we define the \textit{energetic} functional
%\[
%E(t,\psi):=\int_\Omega \left(\thet(x,t)+\frac{|\vv(x,t)|^2}{2}\right)\psi(x)\dx,
%\]
%and we also define the initial energetic functional $E_0(\psi)$ as
%\[
%E_0(\psi):=\int_\Omega \left(\thet_0(x)+\frac{|\vv_0(x)|^2}{2}\right)\psi(x)\dx.
%\]
Here, we introduce the notion of \textit{weak} and \textit{suitable weak} solution to \eqref{eq:v}--\eqref{bcandt=0}. We consider only the dimension $d=3$.
%In the rest of the paper we will assume that  $\Omega\in \mathcal{C}^{0,1}$ and also that $\vv_0,\thet_0$ and $\mathbf{f}$ satisfy \eqref{propertiesdata}.
\begin{definition}[Weak solution]\label{Def1} Let $\Omega\subset\R^3$ be a bounded domain of class $C^{1,1}$  and let $(0, T)$ with $T>0$ be the time interval. Let $\vv_0,\thet_0$ and $\mathbf{f}$ be given functions satisfying \eqref{propertiesdata} and let $p$ be given in the interval $(6/5, +\infty)$.
We say that a triplet $(\vv,\thet,\pi)$ is a \textit{weak} solution to the problem \eqref{eq:v}--\eqref{bcandt=0} if
\begin{align}
\vv &\in \mathcal{C}_{\rm weak}(0,T;L^2(\Omega))\cap L^p(0,T;W_{\mathbf{n},\diver}^{1,p})\cap L^2(0,T; L^2(\partial\Omega)^3), \label{space-v}\\
%\partial_t \vv &\in  L^{p'}(0,T;W_{\mathbf{n},\diver}^{-1,p}),\\
\S&\in  L^{p'}(Q) \mbox{ and } \S=\S^*(\thet, \D\vv) \mbox{ for a.a. } (t,x),\\
\thet &\in   L^\infty(0, T; L^1(\Omega))\cap L^q(0, T; W^{1,q}(\Omega)) \mbox{ for any } q\in \left[1, \frac{5}{4}\right), \\
\thet &\in L^{q}(Q)\mbox{ for any } q\in \left[1, \frac{5}{3}\right), \  \thet(t, x)\geq 0 \ \text{for a.a. } (t, x),\\
\eta& \in L^\infty(0, T; L^1(\Omega)) \cap L^2(0, T; W^{1,2}(\Omega))  \mbox{ and } \eta=\log\thet \mbox{ for a.a. } (t,x), \\
\pi &\in L^{q'}(Q) \mbox{ with } q= \max \left\{p, \frac{5p}{5p-6}\right\} \text{ and }  \int_{\Omega}\pi(t,x)\dx=0 \quad \text{for a.a. } t,
\end{align}
fulfills the following weak formulations: The linear momentum equation~\eqref{eq:v} is satisfied in the following sense
\begin{equation}
\label{momentum}
\begin{split}%\begin{array}{l}\displaystyle\vspace{6pt}
-\int_0^T \int_\Omega \vv\cdot \partial_t\boldsymbol\varphi \dx\dt +\int_0^T\int_\Omega (\S - \vv\otimes\vv):\D\boldsymbol\varphi \dx\dt +\alpha\int_0^T\int_{\partial\Omega}\vv\cdot\boldsymbol\varphi \,{\rm d}\sigma_x \dt\\
=\int_0^T\int_\Omega\pi\diver \boldsymbol\varphi +  \thet\, \mathbf{f}\cdot\boldsymbol\varphi  \dx\dt + \int_\Omega \vv_0\cdot \bphi(0)\dx
\end{split}
\end{equation}
for any $\bphi\in \mathcal{C}_0^\infty([0, T); W_{\mathbf{n}}^{1, q}\cap L^{\infty}(\Omega))\cap L^2((0,T)\times\partial\Omega)$ with $q= \max\{p, \frac{5p}{5p-6}\}$;\\
The global energy inequality holds in the following sense
\begin{equation}\label{weak:energyeq}%\begin{array}{l}\displaystyle\vspace{6pt}
-\int_0^T \int_\Omega \left(\frac{|\vv|^2}{2}+\thet\right) \,\partial_t \varphi \dx\dt+ \alpha \int_0^T \int_{\partial\Omega} |\vv|^2 \varphi \,{\rm d} \sigma_x \dt \le \int_\Omega \left(\frac{|\vv_0|^2}{2}+\thet_0\right)\, \varphi(0)\dx %\\\displaystyle\vspace{6pt} \hfill
%\mbox{ for any } \varphi\in \mathcal{C}_0^\infty([0,T)),
%\end{array}
\end{equation}
for any  nonnegative  $\varphi\in \mathcal{C}_0^\infty([0,T))$;\\
The entropy inequality~\eqref{ineq:entr} is satisfied as
\begin{equation}
\label{eta-entropy}
\begin{split}
&-\intTO{\eta\, \dert\varphi}  - \intTO{\eta\,\vv\cdot \nabla\varphi}  + \intTO{\kappa(\thet)\nabla\eta\cdot\nabla\varphi} \\
&\qquad \geq \intTO{\frac{\S:\D{\vv}}{\thet}\,\varphi} + \intTO{\kappa(\thet)\,|\nabla \eta|^2\, \varphi} + \intO{(\log\thet_0)\, \varphi(0)}
\end{split}
\end{equation}
for any nonnegative $\varphi \in \mathcal{C}_0^\infty([0, T); W^{1,\infty}({\Omega}))$.
The initial conditions are attained in the following sense
\begin{align}\label{INC}
%E(t,\psi)&\in \mathcal{C}([0,T]) \quad \text{and} \quad \lim_{t\to 0^+}E(t,\psi)=E_0(\psi) \  \text{for any } \psi\in\mathcal{C}(\overline{\Omega}),\\
&\lim_{t\to 0_+}\left(\|\vv(t)-\vv_0\|_2 + \|\thet(t) -\thet_0\|_1\right)=0.
\end{align}
\end{definition}
The above definition fulfills the basic assumption on the consistency, i.e. if we have a weak solution that is in addition smooth then it is also the classical solution, we refer here e.g. to the classical book~\cite{Feireisl}, where such approach is justified. On the other hand, if we want to study further properties of the solution, for example the stability, we usually require more refined notion of the solution, namely the suitable weak solution. However, it also requires more assumption on the growth parameter $p$.
\begin{definition}[Suitable weak solution]\label{Def2}
Let $\Omega\subset\R^3$ be a bounded domain of class $C^{1,1}$  and let $(0, T)$ with $T>0$ be the time interval. Let $\vv_0,\thet_0$ and $\mathbf{f}$ be given functions satisfying \eqref{propertiesdata} and let $p\in (9/5, +\infty)$ be given. We say that a triplet $(\vv,\thet,\pi)$ is a \textit{suitable weak} solution to the problem \eqref{eq:v}--\eqref{bcandt=0} if Definition~\ref{Def1} is satisfied with \eqref{weak:energyeq} replaced by
\begin{equation}\label{weak:energyeq2}
\begin{split}
\int_0^T \int_\Omega -\left(\frac{|\vv|^2}{2}+\thet\right) \,\partial_t \varphi -\left(\vv\left(\frac{|\vv|^2}{2}+\thet+\pi\right)+\q-\S \vv\right)\, \cdot\nabla\varphi \dx\dt \\
+ \alpha \int_0^T \int_{\partial\Omega} |\vv|^2 \varphi \,{\rm d} \sigma_x \dt  = \int_\Omega \left(\frac{|\vv_0|^2}{2}+\thet_0\right)\, \varphi(0)\dx,
\end{split}
\end{equation}
which is valid for any $\varphi\in \mathcal{C}_0^\infty([0,T);W^{1,\infty}({\Omega}))$.
Moreover, we require that \eqref{ineq:thet} is satisfied in the following sense
\begin{equation}\label{temperature-ineq}
\intTO{ -\thet \, \partial_t \varphi -(\thet\vv+\q)\cdot\nabla\varphi} \geq \intTO{ \S:\D\vv\,\varphi - \thet (\vv\cdot\mathbf{f}) \, \varphi} + \intO{\thet_0\,\varphi(0)},
\end{equation}
for any nonnegative  $\varphi\in \mathcal{C}_0^\infty([0,T);W^{1,\infty}({\Omega}))$.
\end{definition}

Next, we formulate the main theorem of this paper.
\begin{theorem}\label{thm:mainthem}
Let $\Omega\subset\R^3$ be a bounded domain with  $\mathcal{C}^{1,1}$ boundary. Assume that  $\S^*$ and $\kappa$ satisfy \eqref{k} and \eqref{nu} with $p>6/5$. Then for any data $\vv_0,\thet_0, \mathbf{f}$ fulfilling \eqref{propertiesdata}, there exists a
weak solution to \eqref{eq:v}--\eqref{bcandt=0} in the sense of Definition~\ref{Def1}. Moreover, if $p>8/5$ then \eqref{weak:energyeq} holds with the equality sign. In addition, if $p>9/5$ then there exists a suitable weak solution in sense of Definition~\ref{Def2}. Furthermore, if $p\ge 11/5$, then \eqref{eta-entropy} and \eqref{temperature-ineq} holds with equality sign and the following is true
\begin{equation}\label{prst}
\limsup_{m\to \infty}\int_{Q}\frac{m|\nabla \thet|^2}{\thet^2} \chi_{\{\thet >m\}}\dx \dt =0.
\end{equation}
\end{theorem}

Note that  in above Theorem~\ref{thm:mainthem} we assume a stronger assumption on the boundary, namely $\Omega\in \mathcal{C}^{1,1}$. The reason is the necessity of having a~priori estimates of the pressure $\pi$ that appears in \eqref{weak:energyeq} and cannot be omitted by using divergence-free functions as test functions as it is usual in studying Navier--Stokes equations without the temperature. At this, we would like to discuss the main novelty of the paper. It seems that the only relevant existence result are due to \cite{NaWo10,NaPoWo12}, where the authors treated the same system but with $\S^*$ being linear with respect to the velocity gradient, i.e. the case $p=2$, so the result of this paper is much more general. Second, in \cite{NaPoWo12}, the authors treated only the steady case and in \cite{NaWo10}, the authors were not able to show the validity of \eqref{weak:energyeq2}, i.e. they did not show the existence of a suitable weak solution, which is the main weak point in their result. We also refer to \cite{KiCa20,Ki22}, which is an extension of \cite{NaPoWo12} to more general boundary conditions, but deal only with the steady case.  In our setting, we are able to prove the existence of a suitable weak solution. In addition, for $p\ge 11/5$, we obtain the internal energy equality. Furthermore, in spirit of results \cite{AbBuKa19,AbBuKa22, AbBuKa22b}, we see that the existence result obtained in this paper is the starting point for studying the stability analysis for the underlying problem. We would like to remind that there are naturally appearing numbers like $6/5$, $8/5$, $9/5$, $11/5$, which are dictated by the nature of the problem, and these borderline are usual in the theory for non-Newtonian models of heat conducting incompressible fluids. Adding the dissipative heating to the system does not bring any change in these borderlines. The only change, but rather essential, is the way how the uniform estimates are obtained, which makes the result highly nontrivial extension of works~\cite{AbBuKa19,AbBuKa22} (see also further references therein). In addition, \eqref{prst} as well as energy equalities valid for~$p\ge 11/5$ seem to be the essential assumption to obtain the stability result for arbitrary weak solution, see~\cite{AbBuKa22, AbBuKa22b} or \cite{AbBuKa19} where the same system but without dissipative heating and in dimension two is treated.

The proof is split into several steps. In Section \ref{s:defapporx}, we introduce an approximation, where the nonlinear convective terms are truncated by an auxiliary cut-off function. The existence of solutions for the $k$-approximation, is for the sake of completeness and clarity included in Appendix \ref{s:appendix}. Then, in Section \ref{ss:uniform}, we derive estimates that are uniform with respect to $k$-parameter. Finally, letting $k\to +\infty$ in Section \ref{ss:limit},  we complete the proof of Theorem~\ref{thm:mainthem}.

\section{Definition of approximating systems and their solutions}\label{s:defapporx}
We start this part with definition of auxiliary cut-off functions. For any arbitrary natural number $k\geq 1$, we define
\begin{equation}\label{def:Tk}
\Tk(z) = {\rm sign} (z) \min\{k, |z|\} \mbox{ for any } z\in \R,
\end{equation}
and a function $g_k:\mathbb{R}^{+}\rightarrow [0,1]$ such that it is continuous and satisfies
\begin{equation}
\gk(z)=\begin{cases}
1 &\mbox{ if } z< k,\\
0 &\mbox{ if } z> 2k.
\end{cases}
\end{equation}

Finally, we introduce an auxiliary function $\eta$, which is used in the proof of attainment of the initial data. Let $T > 0$ be given, and let $0 < \varepsilon \ll 1$ and $t \in (0, T-\varepsilon)$ be arbitrary. Consider $\eta \in \mathcal{C}^{0,1}([0, T])$ as a piece-wise linear function of three parameters, such that
\begin{equation}\label{def:eta}
\eta(\tau)=\begin{cases}
1  & \text{if }\tau\in[0,t),\\
1+\frac{t-\tau}{\varepsilon} & \text{if } \tau\in[t,t+\varepsilon),\\
0 & \text{if } \tau \in [t+\varepsilon,T].
\end{cases}
\end{equation}

%\subsection{$k$--approximation}
We define an approximative problem $\mathcal{P}^{k}$ (for simplicity we write $(\vv,\pi,\thet)$ instead of $\vv^{k},\pi^{k},\thet^{k}$) such that we truncate the convective term (in order to be able to use the Minty method), we truncate the source term (in order to have proper estimates at the beginning) and we also modify the boundary conditions (to avoid problem with low integrability). More precisely, we consider the problem:
\begin{align}
\label{def:divapprox}
\diver \vv &= 0,\\
\label{def:vapprox}
\dert \vv + \diver(\vv \otimes \vv \, g_{k}(|\vv|^2)) - \diver \S + \nabla \pi &= \T_k(\thet)\mathbf{f}, \\
\label{def:thetapprox}
\dert\thet+\diver(\T_k(\thet) \, \vv) + \diver \vc{q} &= \S:\D -\T_k(\thet) (\vv\cdot\mathbf{f}),
\end{align}
in $(0,T)\times\Omega$ complemented with the boundary conditions
\begin{equation}
\vv\cdot\mathbf{n}=0, \qquad \alpha\vv_{\tau}g_{k}(|\vv_{\tau}|)+[\S\mathbf{n}]_{\tau}=0,\qquad \nabla\thet\cdot\mathbf{n}=0 \qquad\text{on } \partial\Omega,
\end{equation}
and the following initial conditions
\begin{equation}\label{initial-k}
\vv(0) = \vv_0, \qquad \thet(0) = \thet_0>0 \qquad \text{in } \Omega.
\end{equation}

For this problem, we have the following existence result, which is formulated in any dimension $d$. Note that the result is dimension-independent due to the presence of truncation functions.
\begin{lemma}\label{mainlemma}
Let $\Omega\subset\R^d$ be a bounded domain with  $\mathcal{C}^{1,1}$ boundary. Assume that  $\S^*$ and $\kappa$ satisfy~\eqref{k} and~\eqref{nu} with $p>2d/(d+2)$. Then for any $k\in \mathbb{N}$ and  any data $\vv_0,\thet_0, \mathbf{f}$ fulfilling \eqref{propertiesdata}, there exists a triplet $(\vv,\thet,\pi)=(\vv^k,\thet^k,\pi^k)$ satisfying
\begin{align}
\vv &\in  \mathcal{C}(0,T;L_{\mathbf{n},\diver}^2)\cap L^p(0,T;W_{\mathbf{n},\diver}^{1,p}),\\
\partial_t \vv &\in  L^{p'}(0,T;W_{\mathbf{n}}^{-1,p}),\\
\thet &\in L^\infty(0, T; L^1(\Omega))  \quad \text{ and } \quad \thet(t,x)\geq 0 \quad \text{for a.a. } (t, x)\in (0, T)\times \Omega,\\
\log \thet& \in L^\infty(0, T; L^1(\Omega)),\\
\pi &\in L^{p'}(0,T;L^{p'}(\Omega)) \quad \text{ and } \quad \int_{\Omega}\pi(t,x)\dx=0 \quad \text{for a.a. } t\in (0,T),\\
(1+\thet)^{\frac{1-\varepsilon}{2}}&\in L^2(0,T; W^{1,2}(\Omega))  \mbox{ for any } \varepsilon >0
\end{align}
and
\begin{equation}
\label{entp}\lim_{m\to \infty} \int_{(0,T)\times\Omega\cap \{\thet \ge m\}} \frac{m|\nabla \thet|^2}{\thet^2} \dx \dt =0;
\end{equation}
attaining the initial conditions \eqref{initial-k} in the following sense
\begin{equation}\label{limt0approx}
\lim_{t\to 0_+}\left(\|\vv-\vv_0\|_2+ \|\thet-\thet_0\|_1\right)=0;
\end{equation}
satisfying equation \eqref{def:vapprox} in the following sense:  for any $\boldsymbol\varphi \in W^{1,p}_{\n}$ and  for a.a. $t\in (0, T)$ there holds
\begin{equation}\label{vel}
\begin{split}
\langle \partial_t\vv, \boldsymbol\varphi \rangle - \intO{\gk(|\vv|^2)\, (\vv\otimes\vv) : \nabla \boldsymbol\varphi} + \intO{\S:\D\boldsymbol\varphi} + \alpha  \int_{\partial\Omega} \alpha\vv_{\tau}g_{k}(|\vv_{\tau}|)\cdot \boldsymbol\varphi \,{\rm d}\sigma_x&\\
= \intO{ \Tk(\thet) \,\mathbf{f}\cdot\boldsymbol\varphi + \pi \diver \boldsymbol\varphi},&
\end{split}
\end{equation}
and satisfying \eqref{def:thetapprox}  in the following sense: for any $f\in \mathcal{C}^2(\R)$ satisfying $f''\in \mathcal{C}_0(\R)$, for any $\varphi \in W^{1,2}(\Omega)\cap L^\infty(\Omega)$ and for a.a. $t\in (0, T)$ there holds
\begin{equation}\label{temp}
\begin{split}
\langle \partial_t f(\thet), \varphi\rangle - \intO{f(\Tk(\thet))\, \vv\cdot\nabla\varphi} + \intO{f'(\thet)\kappa(\thet)\nabla\thet \cdot \nabla \varphi } + \intO{f''(\thet)\kappa(\thet)|\nabla\thet|^2\,  \varphi }&\\
= \intO{f'(\thet)\,\S:\D\vv \,\varphi} - \intO{\Tk(\thet)f'(\thet)\,\vv\cdot \mathbf{f} \,\varphi}.&
\end{split}
\end{equation}
\end{lemma}
%
%
%
%
%\begin{lemma}\label{mainlemma}
%Let $\Omega\subset\R^3$ be a bounded domain with  $\mathcal{C}^{1,1}$ boundary. Assume that  $\S^*$ and $\kappa$ satisfy \eqref{k} and \eqref{nu} with $p>6/5$. Then for any data $\vv_0,\thet_0, \mathbf{f}$ fulfilling \eqref{propertiesdata}, there exists a
%weak solution to \eqref{def:divapprox}--\eqref{initial-k}
%in the sense of Definition \ref{Def-k}.
%\end{lemma}
\begin{proof}
The complete proof is presented at the Appendix~\ref{s:appendix} for the most important case $d=3$. For other dimensions the proof is however almost identical.
\end{proof}
We would like to emphasize here, that the above existence result is in fact very strong. Although the equation~\eqref{def:vapprox} is satisfied in the classical weak sense~\eqref{vel}, the equation~\eqref{def:vapprox} is satisfied in the renormalized weak sense as~\eqref{temp}. This enables us to deduce the proper uniform estimates rigorously.

\section{Limit in the approximating system}\label{s:ktoinf}
In the previous section we established the existence of a weak solution to the $k$-approximating system \eqref{def:divapprox}--\eqref{def:thetapprox}. Key $k$-uniform estimates and the limits as $k\to +\infty$ are derived in this section. We focus only on dimension $d=3$, the proof for $d=2$ is in fact even easier.

\subsection{Uniform estimates}\label{ss:uniform}
For $(\vv,\thet,\pi)=(\vv^k,\thet^k,\pi^k)$ we derive estimates that are uniform with respect to $k$-parameter (the relevant quantities are then bounded by a generic constant $C$, where $C(\|\mathbf{f}\|_{\infty}, \|\vv_0\|_2, \|\thet_0\|_1, \|\log \thet_0\|_1)$).

We set $\boldsymbol\varphi=\vv$ in \eqref{vel} and in \eqref{temp} we set $\varphi=1$ and $f(\thet)=\thet$. Summing both identities and using the fact that $\diver \vv = 0$ to eliminate convective terms\footnote{Compare it with very similar computations in Appendix.} we deduce
\begin{equation*}
\frac{1}{2}\frac{{\rm d}}{\dt}\|\vv\|_2^2 + \frac{{\rm d}}{\dt} \|\thet\|_1 + \alpha \|\vv\sqrt{g_k(|\vv|)}\|_{2,\partial\Omega}^2 = 0.
\end{equation*}
Integration with respect the time variable then leads to
\begin{equation}\label{A1}
\sup_{t\in(0, T)} \left(\|\vv(t)\|_2^2 + \|\thet(t)\|_1\right) +\alpha \int_0^T \|\vv\sqrt{g_k(|\vv|)}\|_{2,\partial\Omega}^2\dt \leq C(\|\vv_0\|_2, \|\thet_0\|_1).
\end{equation}
Next, for any $\varepsilon>0$ we set $f(\thet):= \log (\thet + \varepsilon)$  and $\varphi:= 1$  in \eqref{temp}, to deduce
\begin{equation*}
\frac{{\rm d}}{\dt} \intO{\log(\thet + \varepsilon)}= \intO{\frac{\kappa(\thet)|\nabla\thet|^2}{(\thet + \ep)^2}} + \intO{\frac{\S:\D\vv}{\thet + \ep}} - \intO{\frac{\Tk(\thet)}{\thet + \ep } \vv\cdot \mathbf{f}}.
\end{equation*}
Integrating this identity over the time interval $(0, t)$, it follows
\begin{equation*}
\begin{split}
&\int_0^t\intO{\frac{\kappa(\thet)|\nabla\thet|^2}{(\thet + \ep)^2}}\dt + \int_0^t\intO{\frac{\S:\D\vv}{ \thet + \ep }}\dt + \int_{\{\thet(t)+\ep \leq 1\}} |\log(\thet(t) +\ep)|\dx \\
&= \int_0^t\intO{\frac{\Tk(\thet)}{\thet + \ep }\vv\cdot \mathbf{f}}\dt +  \int_{\{\thet(t)+\ep >1\}}\log(\thet(t) +\ep)\dx\ - \intO{\log(\thet_0 +\ep)},
\end{split}
\end{equation*}
and taking the supremum over $t\in (0, T)$, we have
$$
\begin{aligned}
\intTO{\frac{\kappa(\thet)|\nabla\thet|^2}{(\thet + \ep)^2}} + \intTO{\frac{\S:\D(\vv)}{ \thet + \ep }} + \sup_{t\in(0,T)}\|\log{(\thet(t) + \ep)}\|_1 \\
\leq  C\intTO{|\vv\cdot \mathbf{f}|}+C\intO{|\thet(t)|} +\ep C\, \text{meas}(\Omega)+ \|\log \thet_0\|_1.
\end{aligned}
$$
Then employing \eqref{A1}, the fact that $\log\thet_0\in L^1(\Omega)$ and taking the limit as $\ep\to 0$ we get the following $k$-independent estimate
\begin{equation}\label{entropy-L81}\begin{split}
\sup_{t\in(0,T)}\|\log{(\thet(t))}\|_1+ &\intTO{\frac{\kappa(\thet)|\nabla\thet|^2}{\thet^2}} + \intTO{\frac{\S:\D\vv}{ \thet }}\\
 &\qquad \leq C(\|\vv_0\|_2, \|\thet_0\|_1, \|\log\thet_0\|_1, \|\mathbf{f}\|_{\infty,Q}).\end{split}
\end{equation}
In order to improve the uniform bound for the temperature, we fix arbitrary $\sigma\in (0, 1)$ and set $f(\thet)=\thet^{\sigma}$ and $\varphi=1$ in \eqref{temp} to obtain
\begin{equation*}
\frac{1}{\sigma}\frac{{\rm d}}{\dt} \intO{\thet^{\sigma}} + (\sigma-1)\intO{\kappa(\thet)\frac{|\nabla\thet|^2}{\thet^2}\thet^{\sigma}} = \intO{\thet^{\sigma-1} \S:\D\vv} -  \intO{\frac{\Tk(\thet)}{\thet}\, \thet^{\sigma}\,\vv\cdot \mathbf{f}}.
\end{equation*}
Integrating it over time and using the fact that $\sigma\in (0,1)$ and the uniform estimate \eqref{A1}, we deduce
\begin{equation}\label{thetaalpha}
\begin{split}
\intTO{\thet^{\sigma-1} \S:\D\vv} + &\intTO{\kappa(\thet)\frac{|\nabla\thet|^2}{\thet^2}\thet^{\sigma}}\\
&\leq C(\sigma, \|\vv_0\|_2, \|\thet_0\|_1,\|\mathbf{f}\|_{\infty,Q})\left(1+\intTO{\thet^{\sigma} |\vv|}\right).
\end{split}
\end{equation}
%For $\alpha=1/2$ we can use the H\"{o}lder inequality and \eqref{A1} to deduce that
%\begin{equation}
%\intTO{\thet^\frac{1}{2} |\vv|}\leq \|\thet\|_1^\frac{1}{2}\|\vv\|_2\leq C(\|\vv_0\|_2, \|\thet_0\|_1).
%\end{equation}
To bound the term on the right hand side, we recall the interpolation inequality (here we consider~$d=3$)
$$
\|\thet^{\frac{\sigma}{2}}\|_4^2 \leq  C(\Omega)\|\thet^{\frac{\sigma}{2}}\|_2^{\frac12} \|\thet^{\frac{\sigma}{2}}\|_{1,2}^{\frac32}
$$
and using also the H\"{o}lder inequality, the a~priori bound~\eqref{A1} and the assumption~\eqref{k}, we get the estimate
\begin{equation}\label{theta4}
\begin{split}
&C(\sigma, \|\vv_0\|_2, \|\thet_0\|_1,\|\mathbf{f}\|_{\infty,Q})\left(1+\int_0^T\|\vv\|_2\|\thet^{2\sigma}\|_1^{\frac{1}{2}}\dt\right)\leq C\left(1+\int_0^T\|\thet^{\frac{\sigma}{2}}\|_4^{2}\dt\right)\\
&\quad \le C\left(1+\int_0^T\|\thet^{\frac{\sigma}{2}}\|_2^{\frac12} \|\thet^{\frac{\sigma}{2}}\|_{1,2}^{\frac32}\dt\right) \le  \left(C+ \int_0^T\int_{\Omega} C(\sigma,\underline{\kappa})\thet^{\sigma} + \frac{\underline{\kappa}}{2}|\nabla \thet^{\frac{\sigma}{2}}|^2 \dx \dt \right)\\
&\quad \le C(\sigma,\Omega, \underline{\kappa}, \|\vv_0\|_2, \|\thet_0\|_1,\|\mathbf{f}\|_{\infty,Q}) +\int_0^T\int_{\Omega} \frac{\kappa(\thet)|\nabla \thet|^2}{2\thet^{2-\sigma}}\dx \dt.
\end{split}
\end{equation}
%
%we get (recalling \eqref{A1} and the fact that $|\nabla \thet^{\frac{\alpha}{2}}|^2 \sim \frac{|\nabla \thet|^2}{\thet^2}\thet^{\alpha}$)
%\begin{equation}
%\int_0^T\|\thet^{\frac{\alpha}{2}}\|_4^2\dt\leq \frac{\varepsilon}{2} \int_0^T \|\thet^{\frac{\alpha}{2}}\|_{1,2}^2\dt + C(\alpha, \lambda, %\|\thet_0\|_1).
%\end{equation}
Collecting \eqref{thetaalpha} and \eqref{theta4} we deduce the following estimates that are uniform with respect to~$k\in \N$
\begin{equation}\label{gradient-theta-alpha}
 \intTO{\kappa(\thet)\frac{|\nabla\thet|^2}{\thet^2}\thet^{\sigma}}\leq C(\sigma, \underline{\kappa},\Omega,\|\vv_0\|_2, \|\thet_0\|_1,\|\mathbf{f}\|_{\infty,Q}).
\end{equation}
Consequently, we deduced that for arbitrary $\sigma \in (0,1)$ there holds
\begin{equation}\label{theta-alpha}
\thet^{\frac{\sigma}{2}} \mbox{ is uniformly bounded with respect to } k\in \N \mbox{ in } L^2(0, T; W^{1,2}(\Omega))\cap L^\infty(0, T; L^2(\Omega)).
\end{equation}

Next, we derive the final estimates for $\thet$. Using the interpolation inequality
\begin{equation}\label{int3d}
\|z\|_{\frac{10}{3}}^{\frac{10}{3}}\leq C(\Omega)\|z\|_2^{\frac43}\|z\|_{1,2}^{2}
\end{equation}
on the function $z:=\thet^{\sigma}$, with $\sigma:=\frac{3q}{5}$ for $q<5/3$, we obtain from \eqref{theta-alpha} that
\begin{equation}\label{thet-uniform}
\intTO{\thet^q} = \intTO{\left|\thet^{\frac{\sigma}{2}}\right|^{\frac{10}{3}}} \leq C(q, \underline{\kappa},\Omega,\|\vv_0\|_2, \|\thet_0\|_1,\|\mathbf{f}\|_{\infty,Q}) \mbox{ for any } q\in [1, 5/3).
\end{equation}
Similarly, combining \eqref{gradient-theta-alpha} and \eqref{thet-uniform} and using the H\"{o}lder inequality, we get
\begin{equation}\label{nn-thet-uniform}
\begin{split}
\intTO{|\nabla\thet|^r}&=\intTO{\left(\frac{|\nabla\thet|^2}{\thet^{2-\sigma}}\right)^{\frac{r}{2}}\thet^{\frac{r(2-\sigma)}{2}}}\\
&\le \left(\intTO{\frac{|\nabla\thet|^2}{\thet^{2-\sigma}}}\right)^{\frac{r}{2}} \left(\intTO{\thet^{\frac{r(2-\sigma)}{(2-r)}}}\right)^{\frac{2-r}{2}} \\
& \leq C(r, \underline{\kappa},\Omega,\|\vv_0\|_2, \|\thet_0\|_1,\|\mathbf{f}\|_{\infty,Q}) \mbox{ for any } r\in [1, 5/4).
\end{split}
\end{equation}
The last bound can be viewed by setting
$$
q:=\frac{5-r}{3(2-r)} <\frac53 \qquad \sigma:=1-\frac{5-4r}{3r} <1 \textrm{ provided that }r<\frac54.
$$
Finally, we use that  $\thet^{\frac{\sigma}{2}}$ is uniformly bounded in $L^2(0, T; L^6(\Omega))$ and consequently $\thet^\sigma$ is uniformly bounded in $L^1(0, T; L^3(\Omega))$, combine it with the interpolation inequality (valid for some $\sigma \in (0,1)$ and $\lambda\in (0,1)$)
$$
\|\thet\|_2\leq C\|\thet\|_1^\lambda \|\thet\|_{3\sigma}^{1-\lambda}\mbox{ with } \lambda=\lambda(\sigma)\in (0,1)
$$
and use the uniform estimates \eqref{A1} and \eqref{theta-alpha}, we see that
\begin{equation}\label{endtheta}
\int_0^T\|\thet\|_2\dt\leq C \int_0^T \|\thet\|_{3\sigma}^{1-\lambda}\dt = C \int_0^T \|\thet^{\frac{\sigma}{2}}\|_{6}^\frac{2(1-\lambda)}{\sigma}\dt \le  C(\underline{\kappa},\Omega,\|\vv_0\|_2, \|\thet_0\|_1,\|\mathbf{f}\|_{\infty,Q}).
\end{equation}
%It is easy to verify that there exists an interval of $\alpha$ included in $(0,1)$ such that $0<(1-\lambda)/\alpha <1$, then the H\"{o}lder inequality implies that
%\begin{equation}
%\int_0^T\|\thet\|_2\dt \mbox{ is uniformly bounded.}
%\end{equation}

Having the estimate~\eqref{endtheta}, we can now proceed with further bounds on the velocity field. Setting $\boldsymbol\varphi:=\vv$ in \eqref{vel} and integrating in time, it yields that
\begin{equation}\label{kinetic-energy}
\begin{split}
\sup_{t\in (0, T)}\|\vv(t)\|^2_2+\intTO{\S:\D\vv}&\leq C\left(\|\vv_0\|^2_2 + \intTO{\thet |\vv|}\right)\leq  C\left(1+\int_0^T \|\thet\|_2\dt\right) \\
&\le C(\underline{\kappa},\Omega,\|\vv_0\|_2, \|\thet_0\|_1,\|\mathbf{f}\|_{\infty,Q}).
\end{split}
\end{equation}
%and finally that
%\begin{equation}\label{kinetic-energy}
%\sup_{t\in (0, T)}\|\vv(t)\|_2 + \intTO{\S:\D\vv}\leq C( \alpha,\|\vv_0\|_2,  \|\thet_0\|_1,\|\log\thet_0\|_1).
%\end{equation}
As consequence it follows from \eqref{nu2} that
\begin{equation}\label{SD-uniform}
\intTO{|\S|^{p'}} + \intTO{|\D\vv|^p} \leq C(\underline{\kappa}, \underline{\nu}, \overline{\nu},\Omega,\|\vv_0\|_2, \|\thet_0\|_1,\|\mathbf{f}\|_{\infty,Q}).
\end{equation}
The interpolation inequality
$$
\|\vv\|_{\frac{5p}{3}} \leq C \|\vv\|_2^{\frac{2}{5}}\|\nabla\vv\|_{p}^{\frac{3}{5}} +\|\vv\|_2
$$
together with the Korn inequality~\eqref{eq:Korna0} and the uniform estimates \eqref{kinetic-energy}, \eqref{SD-uniform} ensure that
\begin{equation}\label{vvk}
\int_0^T\|\vv\|_{\frac{5p}{3}}^{\frac{5p}{3}}\dt \leq C(\underline{\kappa}, \underline{\nu}, \overline{\nu},\Omega,\|\vv_0\|_2, \|\thet_0\|_1,\|\mathbf{f}\|_{\infty,Q}).
\end{equation}

We finish this part by introducing the estimates on the pressure. It follows from \eqref{vel} that for all $\varphi \in W^{2,p'}(\Omega)$ satisfying $\nabla \varphi \cdot \mathbf{n} =0$ on $\partial \Omega$, and for almost all time $t\in (0,T)$ there holds (see also \eqref{eq:shortpi})
$$
\begin{aligned}
\intO{\pi \, \Delta\varphi}&= \intO{\S: \nabla^2 \varphi} - \intO{g_k(|\vv|^2) 	\, \vv\otimes\vv: \nabla^2 \varphi} \\
&\quad + \alpha \int_{\partial\Omega} g_k(|\vv|)\vv\cdot \nabla\varphi {\rm d}\sigma_x - \intO{\T_k(\thet)\mathbf{f}\cdot \nabla \varphi }.
\end{aligned}
$$
In addition, recall that
$$
\int_{\Omega} \pi \dx =0 \qquad \textrm{ for a.a. } t\in (0,T).
$$
Then, we use the fact that $\Omega\in \mathcal{C}^{1,1}$ and the theory for Laplace equation (see also \cite{BuMR,BuFeMa09,BuMaRa09,BuGwMaSw12} for details) and obtain that
\begin{equation*}
\|\pi\|_{z'}\leq C(\|\S\|_{p'} + \||\vv|^2\|_{\frac{5p}{6}} + \|\vv g_k(|\vv|)\|_{2, \partial\Omega} + \|\f\|_\infty\|\thet\|_{\frac{5}{4}}) \mbox{ with } z':=\min \left\{p', \frac{5p}{6}, \frac{5}{3} \right\}.
\end{equation*}
Applying the $z'$-power and integrating the result over $(0,T)$ we finally get
\begin{equation}\label{pressure}
\begin{split}
\int_0^T\|\pi\|_{z'}^{z'}\dt &\leq C\left(\int_0^T 1+\|\S\|_{p'}^{p'} + \|\vv\|_{\frac{5p}{3}}^{\frac{5p}{3}} + \|\vv \sqrt{g_k(|\vv|)}\|^2_{2, \partial\Omega} + \|\f\|^{\frac53}_\infty\|\thet\|^{\frac53}_{\frac{5}{4}}\right)\\
&\le C, \qquad \mbox{ for } z':=\min \left\{p', \frac{5p}{6}, \frac{5}{3} \right\},
\end{split}
\end{equation}
where the last  bound follows from the estimates \eqref{A1}, \eqref{gradient-theta-alpha}, \eqref{SD-uniform} and \eqref{vvk}.

Finally, we recall the estimates on the time derivative. By using the very classical procedure, we can deduce from \eqref{vel} with the help of above uniform estimates  \eqref{A1}, \eqref{gradient-theta-alpha}, \eqref{SD-uniform} and \eqref{vvk} that
\begin{equation}\label{timev}
\int_0^T \|\partial_t \vv^k\|^{z'}_{W_{\mathbf{n}}^{-1,z'}}\dt \le C,
\end{equation}
where $z'$ is defined in \eqref{pressure}. Similarly, considering \eqref{temp} with $f(s):=\log s$, we can use the above estimates \eqref{A1}, \eqref{gradient-theta-alpha}, \eqref{SD-uniform} and \eqref{vvk} to observe that for any $\omega>5$, we have (see e.g. \cite{BuMaRa09} for similar estimate)
\begin{equation}\label{timeth}
\int_0^T \|\partial_t \eta^k\|_{W^{-1,\omega'}}\dt \le C.
\end{equation}

\subsection{Limit as \texorpdfstring{$k\to+\infty$}{k}}\label{ss:limit}
Let us consider $\bphi\in C_0^\infty([0, T); W_{\mathbf{n}}^{1, q}\cap L^{\infty}(\Omega))\cap L^2((0,T)\times\partial\Omega)$ with $q= \max\{p, \frac{5p}{5p-6}\}$ in \eqref{vel}, integrate it over the time interval $(0, T)$, then after the integration by parts in the time derivative term we get
\begin{equation}\label{momentum-k}
\begin{split}
-&\intTO{\vv^k\cdot\partial_t \boldsymbol\varphi} - \intTO{\gk(|\vv^k|^2)\, (\vv^k\otimes\vv^k) : \nabla \boldsymbol\varphi} + \intTO{\S^k:\D\boldsymbol\varphi} \\
&\qquad + \alpha  \int_0^T\int_{\partial\Omega} g_k(|\vv^k|)\vv^k \cdot \boldsymbol\varphi \,{\rm d}\sigma_x\dt
= \intTO{ \Tk(\thet^k) \,\mathbf{f}\cdot\boldsymbol\varphi + \pi^k \diver \boldsymbol\varphi} + \intO{\vv_0^k\cdot\bphi(0)},
\end{split}
\end{equation}
where we abbreviate $\S^k=\S^*(\thet^k, \D\vv^k).$ Next, we consider  $\varphi\in \mathcal{C}_0^\infty([0,T);\mathcal{C}^{\infty}({\Omega}))$ in \eqref{temp},  integrate it over the time interval $(0, T)$, and after the integration by parts with respect to the time variable,  we get
\begin{equation}\begin{split}
-&\intTO{f(\thet^k) \partial_t \varphi} - \intTO{f(\Tk(\thet^k))\, \vv^k\cdot\nabla\varphi}\\
&\qquad  + \intTO{f'(\thet^k)\kappa(\thet^k)\nabla\thet^k \cdot \nabla \varphi }+ \intTO{f''(\thet^k)\kappa(\thet^k)|\nabla\thet^k|^2\,  \varphi }\\
&= \intTO{f'(\thet^k)\,\S^k:\D\vv^k \,\varphi} - \intTO{\Tk(\thet^k)f'(\thet^k)\,\vv^k\cdot \mathbf{f} \,\varphi}+ \intO{f(\thet^k_0) \varphi(0)}.
\end{split}
\end{equation}
Now, we make two special choices of $f$, namely we consider $f(s)=s$ and then $f(s)=\log s$ (such choices can be rigorously justified taking the mollification with compactly supported functions). With the first choice, we deduce
\begin{equation}\label{temperature-k}
\begin{split}
&-\intTO{\thet^k \partial_t \varphi} - \intTO{\Tk(\thet^k)\, \vv^k\cdot\nabla\varphi} + \intTO{\kappa(\thet^k)\nabla\thet^k \cdot \nabla \varphi }\\
&\qquad = \intTO{\S^k:\D\vv^k \,\varphi} - \intTO{\Tk(\thet^k)\,\vv^k\cdot \mathbf{f} \,\varphi}
+ \intO{\thet^k_0\, \varphi(0)}.
\end{split}
\end{equation}
With the second choice, we also use the abbreviation $\eta^k=\log \thet^k$ and we see that
\begin{equation}\label{entropy-k}\begin{split}
&-\intTO{\eta^k \, \partial_t \varphi} - \intTO{\log(\Tk(\thet^k))\, \vv^k\cdot\nabla\varphi} \\
&\qquad + \intTO{\kappa(\thet^k)\nabla\eta^k \cdot \nabla \varphi } - \intTO{\kappa(\thet^k)|\nabla\eta^k|^2\,  \varphi }\\
&\quad = \intTO{\frac{1}{\thet^k}\,\S^k:\D\vv^k \,\varphi} - \intTO{\frac{\Tk(\thet^k)}{\thet^k}\,\vv^k\cdot \mathbf{f} \,\varphi}
+ \intO{\eta^k_0 \varphi(0)}.
\end{split}
\end{equation}
Finally, to get the energy identity, we consider $\varphi \in \mathcal{C}^\infty((0,T)\times \Omega)$ fulfilling $\varphi(T)=0$ and use $\vv^k \varphi$ and $\varphi$ as test functions in \eqref{momentum-k} and \eqref{temperature-k} respectively. Doing this and then  taking the sum of the outcome we deduce that (using also the fact that $\diver \vv^k=0$ and integration by parts\footnote{To evaluate the convective term we proceed as follows
$$
\begin{aligned}
&\intO{\gk(|\vv^k|^2)\, (\vv^k\otimes\vv^k) : \nabla (\vv^k \varphi)}=\intO{\varphi \gk(|\vv^k|^2)\, (\vv^k\otimes\vv^k) : \nabla \vv^k +\gk(|\vv^k|^2)\, (\vv^k\otimes\vv^k) : (\vv^k \otimes \nabla \varphi)}\\
&\quad =\frac12\intO{\varphi \gk(|\vv^k|^2)\, \vv^k \cdot \nabla |\vv^k|^2 +2\gk(|\vv^k|^2)\, |\vv^k|^2 \vv^k \cdot  \nabla \varphi} =\frac12\intO{\varphi  \vv^k \cdot \nabla \mathcal{G}_k(|\vv^k|^2) +2\gk(|\vv^k|^2)\, |\vv^k|^2 \vv^k \cdot  \nabla \varphi}\\
&\quad =\intO{\left(\gk(|\vv^k|^2)\, |\vv^k|^2 -\frac{\mathcal{G}_k(|\vv^k|^2)}{2}\right) \vv^k \cdot  \nabla \varphi}
\end{aligned}
$$}
)
\begin{equation}\label{totalenergy-grad}
\begin{split}
&\int_0^T \int_\Omega -\left(\frac{|\vv^k|^2}{2}+\thet\right) \,\partial_t \varphi \dx \dt + \alpha \int_0^T \int_{\partial\Omega} g_k(|\vv^k|)|\vv^k|^2 \varphi \,{\rm d} \sigma_x \dt\\
&\quad +\int_0^T \int_{\Omega} \left(-\vv^k\left(\frac{2\gk(|\vv^k|^2)\, |\vv^k|^2 -\mathcal{G}_k(|\vv^k|^2)}{2}+\mathcal{T}_k(\thet^k)+\pi^k\right)+\kappa(\thet^k)\nabla\thet^k+\S^k \vv^k\right)\, \cdot\nabla\varphi \dx\dt \\
   &= \int_\Omega \left(\frac{|\vv_0^k|^2}{2}+\thet_0\right)\, \varphi(0)\dx,
\end{split}
\end{equation}
where $\mathcal{G}_k$ is such that $\mathcal{G}_k'(s)=g_k(s)$.  In particular, for any $\varphi\in \mathcal{C}_0^\infty([0,T))$, i.e. $\varphi$ independent of spatial variable,   there holds
\begin{equation}\label{total-energy}\begin{array}{l}\displaystyle\vspace{6pt}
-\int_0^T \int_\Omega \left(\frac{|\vv^k|^2}{2}+\thet^k\right) \,\partial_t \varphi \dx\dt+ \alpha \int_0^T \int_{\partial\Omega} g_k(|\vv^k|)|\vv^k|^2 \varphi \,{\rm d} \sigma_x \dt = \int_\Omega \left(\frac{|\vv_0^k|^2}{2}+\thet_0^k\right)\, \varphi(0)\dx.
%\\\displaystyle\vspace{6pt} \hfill \mbox{ for any } \varphi\in \mathcal{C}_0^\infty([0,T)).
\end{array}
\end{equation}

We want to discuss the limit in formulations \eqref{momentum-k} and \eqref{temperature-k}--\eqref{total-energy}.
By virtue of the uniform estimates \eqref{kinetic-energy}, \eqref{SD-uniform} we can extract a subsequence $(\vv^k,\thet^k, \pi^k, \S^k, \eta^k)$ such that the following convergence results hold
\begin{align}
\vv^k &\rightharpoonup^* \vv && \mbox{ weakly-* in } L^\infty(0, T; L^2(\Omega; \R^3)), \label{Linfty2}\\
 \vv^k &\rightharpoonup \vv &&\mbox{ weakly in } L^p(0, T; W^{1,p}_{\n,\diver}),\label{Lp}\\
 \partial_t \vv^k &\rightharpoonup \partial_t \vv &&\mbox{ weakly in } L^{q'}(0, T; W^{-1,q'}_{\n})\mbox{ with } q'=\min \left\{p', \frac{5p}{6}, \frac{5}{3} \right\},\label{Lptd}\\
 g_k(|\vv^k|)\vv^k&\rightharpoonup \vv &&\mbox{ weakly in } L^2(0,T; L^2(\partial\Omega; \R^3)),\\
 \S^k &\rightharpoonup \S &&\mbox{ weakly in } L^{p'}(Q; \R^{3\times 3}),\label{Lp-S}\\
 \pi^k &\rightharpoonup \pi &&\mbox{ weakly in }L^{q'}(Q) \mbox{ with } q'=\min\{p', \frac{5p}{6}, \frac{5}{3} \},\\
  (\thet^k)^{\sigma}&\rightharpoonup   \overline{(\thet)^{\sigma}} &&\mbox{ weakly in } L^2(0, T; W^{1,2}(\Omega)) \mbox{ for any } \sigma\in (0, {1}/{2}),\label{thet-weak}\\
 \thet^k  &\rightharpoonup \thet  && \mbox{ weakly in } L^q(0, T; W_0^{1, s}(\Omega))  \mbox{ for any } q\in [1, 5/4).\label{grad-thet-weak}\\
 %\eta^k &\rightharpoonup \eta &&\mbox{ weakly in }L^{q}(Q) \mbox{ for any } q\in [1, \infty),\\
 \eta^k  &\rightharpoonup  \eta &&\mbox{ weakly in }L^{2}(0,T; W^{1,2}(\Omega)).\label{gradient-eta}
 \end{align}
Moreover, employing the Aubin--Lions compactness Lemma we deduce that
\begin{align}
 \vv^k&\to\vv &&\mbox{ strongly in }L^q(Q; \R^3) \mbox{ for any $q\in [1, {5p}/{3})$ and a.e. in $Q$}, \label{v-strong}\\
 \vv^k&\to \vv &&\mbox{ strongly in } L^p((0, T; L^1(\partial\Omega)) \mbox{ and a.e. in } (0, T)\times\partial\Omega,\label{v-stb}\\
 \thet^k &\to \thet &&\mbox{ strongly in $L^q(Q)$ for any $q\in [1, 5/3)$ and a.e. in $Q$}.\label{thet-strong}
\end{align}
Consequently, we have that $\eta=\log \thet$ and $\overline{(\thet)^{\sigma}}=(\thet)^{\sigma}$. The above strong and weak convergence results are sufficient to pass to the limit in most term.

However, it is not sufficient for identification of $\S=\S^*(\thet, \D\vv)$. This identification follows from the procedure developed in \cite{DiRuWo}, see also \cite{BuGwMaSw12}. Indeed, there is shown that there exists a nondecreasing sequence of measurable sets $\{Q_n\}_{n=1}^{\infty}$ fulfilling $\lim_{n\to \infty}|Q\setminus Q_n|=0$ such that for every $n\in \mathbb{N}$ we have
\begin{equation}\label{Minty1}
\lim_{k\to \infty} \int_{Q_n} \S^k : (\D\vv^k - \D \vv) \dx \dt = 0.
\end{equation}
Further, using the growth assumption \eqref{nu2}, the convergence result \eqref{thet-strong}, the fact that $\D\vv \in L^{p'}(Q; \R^{3\times 3})$ and the Lebesgue dominated convergence theorem, we obtain
\begin{align}
\S^*(\thet^k, \D\vv) &\to \S^*(\thet, \D\vv) &&\mbox{ strongly in }L^{p'}(Q; \R^{3\times 3}). \label{SS-strong}
\end{align}
Thus, using the monotonicity assumption \eqref{nu1}, the weak convergence result~\eqref{Lp} and~\eqref{SS-strong}, we observe that for any $n\in\mathbb{N}$
\begin{equation}\label{strong1}
\begin{split}
&\lim_{k\to \infty}\int_{Q_n}|(\S^k - \S^*(\thet^k, \D\vv)):(\D\vv^k - \D \vv)| \dx \dt\\
&\quad =\lim_{k\to \infty}\int_{Q_n}(\S^k - \S^*(\thet^k, \D\vv)):(\D\vv^k - \D \vv) \dx \dt \overset{\eqref{Minty1},\eqref{SS-strong}}=0.
\end{split}
\end{equation}
Hence, we have that
\begin{align*}
(\S^k - \S^*(\thet^k, \D\vv)):(\D\vv^k - \D \vv)&\to 0 \mbox{ strongly in }L^{1}(Q_n).
\end{align*}
Thus, it is a simple consequence of \eqref{SS-strong} and \eqref{Lp} that for every $n\in \mathbb{N}$
\begin{align}
\S^k:\D\vv^k &\rightharpoonup  \S^*(\thet, \D\vv)): \D \vv  \mbox{ weakly in }L^{1}(Q_n).\label{L1weak}
\end{align}
Finally, using \eqref{Lp}, \eqref{Lp-S} and \eqref{L1weak}, we have for arbitrary $\mathbf{B}\in L^p(Q; \R^{3\times 3})$ that
$$
0\le (\S^k - \S^*(\thet^k, \mathbf{B})): (\D\vv^k - \mathbf{B})\rightharpoonup (\S - \S^*(\thet, \mathbf{B})): (\D\vv - \mathbf{B})\textrm{ weakly in }L^1(Q_n).
$$
Hence, using the Minty method, i.e. setting $\mathbf{B}:=\D \vv \pm \varepsilon \mathbf{C}$ in the above inequality, dividing by $\varepsilon>0$ and then letting $\varepsilon \to 0_+$, we have
$$
0\le \pm  \left(\S - \S^*(\thet, \D \vv)\right): \mathbf{C} \textrm{ a.e. in  } Q_n.
$$
Since $\mathbf{C}$ is arbitrary and $|Q\setminus Q_n|\to 0$ as $n\to \infty$, we observe from the above inequality that $\S=\S^*(\thet, \D\vv)$.

Having identified the nonlinearity $\S$, we may now focus on the limiting procedure in desired equations. Indeed, it is now easy to use \eqref{Linfty2}--\eqref{thet-strong} and to let $k\to \infty$ in \eqref{momentum-k} to deduce \eqref{momentum}. Here, it is essential that $\vv^k$ converges strongly in $L^{2+\varepsilon}(Q)$, which is due to the assumption $p>6/5$. Next, we let $k\to \infty$ in \eqref{total-energy}. Using \eqref{v-strong} and \eqref{thet-strong}, we can pass to the limit in the first term on the left hand side. For the second term on the left hand side we use \eqref{v-stb} and the~Fatou lemma to obtain the inequality \eqref{weak:energyeq}. In order to obtain the equality sign in~\eqref{weak:energyeq}, we can use the result in~\cite{BuMR}, where it is shown that\footnote{It follows from the following interpolation (assuming $p\le 2$, for $p>2$ it is obvious)
$$
\int_0^T \int_{\partial \Omega} |\vv|^2 \le C\int_0^T \|\vv\|^{\frac{8p-12}{5p-6}}_2 \|\vv\|^{\frac{2p}{5p-6}}_{\frac{5p-6}{2p},2}\dt\le C\int_0^T \|\vv\|^{\frac{8p-12}{5p-6}}_2 \|\vv\|^{\frac{2p}{5p-6}}_{1,p}\dt.
$$
Since $\frac{2p}{5p-6}<p$ for $p>8/5$, we see that the desired compactness follows from the a~priori estimates.
}
\begin{equation}\label{strongbo}
\vv^k \to \vv \textrm{ strongly in } L^2(0,T; L^2(\partial \Omega)^3),
\end{equation}
provided that $p>8/5$. Consequently, we obtain \eqref{weak:energyeq} with the equality sign.

Next, we focus on the energy and the entropy (in)equalities \eqref{eta-entropy} and \eqref{temperature-ineq}. To do so, we first show how to pass to the limit with possibly inequality signs in highest order terms. Let us consider arbitrary nonnegative $\varphi \in L^{\infty}(Q)$. Then using \eqref{nu1}, \eqref{nu2} and the convergence result~\eqref{L1weak}, we deduce that for any $n\in \mathbb{N}$
\begin{equation*}
\begin{aligned}
\liminf_{k\to +\infty}\intTO{\S^k:\D\vv^k\, \varphi}&\ge \liminf_{k\to +\infty}\int_{Q_n}\S^k:\D\vv^k\, \varphi\dx \dt = \int_{Q_n}\S:\D\vv\, \varphi\dx \dt.
\end{aligned}
\end{equation*}
Consequently, letting $n\to \infty$ in the above inequality, we deduce
\begin{equation}
\begin{aligned}
\liminf_{k\to +\infty}\intTO{\S^k:\D\vv^k\, \varphi}&\ge  \intTO{\S:\D\vv\, \varphi}.
\end{aligned}
\label{semi-c1}
\end{equation}
Similarly, using in addition \eqref{thet-strong} and the nonnegativity of $\thet^k$, we deduce
\begin{equation}
\begin{aligned}
\liminf_{k\to +\infty}\intTO{\frac{\S^k:\D\vv^k\, \varphi}{\thet^k}}&\ge  \intTO{\frac{\S:\D\vv\, \varphi}{\thet}}.
\end{aligned}\label{semi-c2}
\end{equation}
Very similarly, we can use the weak lower semicontinuity,  the weak convergence \eqref{gradient-eta}, the pointwise convergence of $\thet^k$  \eqref{thet-strong} and the boundedness of $\kappa(\cdot)$, see~\eqref{k}, there holds
\begin{equation}\label{gradeta}
\intTO{\kappa(\thet)|\nabla \eta|^2 \, \varphi}\leq \liminf_{k\to +\infty} \intTO{\kappa(\thet^k)|\nabla \eta^k|^2 \, \varphi}.
\end{equation}
These convergence results are sufficient to take the limit in \eqref{entropy-k} and to obtain \eqref{eta-entropy}. Then we can proceed to the limit in the equation \eqref{temperature-k} to deduce \eqref{temperature-ineq}, provided we show
\begin{equation}\label{SStv}
\thet^k \vv^k \to \thet \vv \textrm{ stronlgy in }L^1(Q;\R^3).
\end{equation}
To show the above convergence result, we recall the strong convergence results \eqref{v-strong} and \eqref{thet-strong}. Thus, it is sufficient to show that $\thet^k \vv^k $ is uniformly bounded in $L^{1+\varepsilon}(Q)$ for some $\varepsilon >0$. Using the H\"{o}lder inequality and the classical interpolation, we have (here $\sigma\in (0,1)$ will be specified later, but typically it will be almost equal to one)
$$
\begin{aligned}
\int_0^T \int_{\Omega} |\thet^k|^{1+\varepsilon}|\vv^k|^{1+\varepsilon}\dx\dt &\le \int_{0}^T \|\thet^k\|^{1+\varepsilon}_{2(1+\varepsilon)}\|\vv^k\|^{1+\varepsilon}_{2(1+\varepsilon)}\dt  =\int_{0}^T \|(\thet^k)^{\sigma}\|^{\frac{1+\varepsilon}{\sigma}}_{\frac{2(1+\varepsilon)}{\sigma}}\|\vv^k\|^{1+\varepsilon}_{2(1+\varepsilon)}\dt\\
&\le C\int_{0}^T \|(\thet^k)^{\sigma}\|^{\frac{1+\varepsilon}{\sigma}-\frac{3(2+2\varepsilon -\sigma)}{4\sigma}}_{1}
\|(\thet^k)^{\sigma}\|^{\frac{3(2+2\varepsilon -\sigma)}{4\sigma}}_{3}\|\vv^k\|^{1+\varepsilon-\frac{3\varepsilon p}{5p-6}}_{2}
\|\vv^k\|^{\frac{3\varepsilon p}{5p-6}}_{1,p}\dt\\
&\overset{\eqref{A1}}\le C\int_{0}^T \|(\thet^k)^{\sigma}\|^{\frac{3(2+2\varepsilon -\sigma)}{4\sigma}}_{3}
\left(\|\vv^k\|^p_{1,p}\right)^{\frac{3\varepsilon }{5p-6}}\dt \le C,
\end{aligned}
$$
provided that (we are using the uniform bounds coming from \eqref{Lp} and \eqref{thet-weak})
\begin{equation}\label{finish}
\frac{3\varepsilon }{5p-6}+\frac{3(2+2\varepsilon -\sigma)}{4\sigma} \le 1.
\end{equation}
Note that if
$$
\varepsilon < \frac{5p-6}{12},
$$
then we can always find $\sigma \in (0,1)$ such that the  inequality \eqref{finish} is satisfied. As a consequence, we deduce \eqref{SStv}. Thus, we may proceed to the limit also in \eqref{temperature-k} to obtain \eqref{temperature-ineq}.

Next, we want to let $k\to \infty$ in \eqref{totalenergy-grad} to get \eqref{weak:energyeq2}. We already discuss almost all terms except the one $|\vv^k|^2\vv$. Thus, for identification of the limit also in this term, we need that $\vv^k$ is compact at least in $L^3(Q)$. Comparing it with the strong convergence result \eqref{v-strong}, we see that the condition $p>9/5$ is exactly the one guaranteeing the compactness of the velocity in $L^3(Q)$.

Finally, in case that $p\ge 11/5$, one may follow the standard monotone operator theory and to conclude that \eqref{L1weak} holds true even in $L^1(Q)$ (and not only in a subset $Q_n\subset Q$). Thus, we have the weak convergence on the whole set $Q$ and therefore we are able to identify the limits in terms containing $\S^k:\D\vv^k$ with equality signs and consequently, we may consider equalities in \eqref{eta-entropy} and \eqref{temperature-ineq}. The relation \eqref{prst} is proved in the same way as \eqref{entp}, see the computation following \eqref{eq:thetA}.

We also omit the proof of attaining of the initial conditions~\eqref{INC} and refer the reader e.g. to \cite{BuMaRa09}. The proofs of both main theorems are thus complete.

%\eqref{total-energy} and establish that \eqref{momentum}, \eqref{weak:energyeq} and \eqref{eta-entropy} hold for any $p>6/5$.
%\\
%
% We discuss now the limit when $p>9/5$. In this case it happens that the term $\frac{1}{2}|\vv^k|^2\vv^k + \thet^k\vv^k$ is compact in $L^1(Q)$, indeed $3<5p/3$ and $\vv^k$ is bounded in $L^{\frac{5p}{3}}(Q)$ and using that $5p/(5p-3)<5/3$ it holds
%$$ \intTO{|\thet^k\vv^k|}\leq \|\thet^k\|_{L^{\frac{5p}{5p-3}}(Q)} \|\vv^k\|_{L^{\frac{5p}{3}}(Q)}\leq C \mbox{ independently of } k.$$
%
%We are in position to take the limit in \eqref{totalenergy-grad} and employing the established convergences we deduce \eqref{weak:energyeq2}. Finally, using in particular the weak lower semicontinuity property \eqref{semi-c1}, we take the limit in  \eqref{temperature-k} for any $\varphi\geq 0$ then the inequality \eqref{temperature-ineq} follows.
%%\newpage

\appendix

\section{Solvability of the \texorpdfstring{$k$}{k}-approximation - Lemma~\ref{mainlemma}} \label{s:appendix}
This appendix is devoted to the proof of Lemma~\ref{mainlemma}, i.e. we assume that $k>0$ is given and fixed. We provide the complete rigorous proof that is however very much inspired by \cite{BuMaRa09}, which can be used also as a tool for interested reader. We refer also to \cite{AbBuKa22,AbBuKa22b}, where the existence, the stability and the convergence to the equilibria are studied in details.

We proceed as follows. First, we introduce the two-level Galerkin approximation: one for the velocity field and the second one for the temperature. Then, we pass to the limit in the temperature equation and derive all necessary a~priori estimates. Finally, we pass to the limit also in the equation for the velocity and complete the proof. Note that because of the presence of the cut-off functions, the proof is relatively standard and we use just the monotone operator theory together with the relatively classical approach for parabolic equations with $L^1$ data.

For the sake of simplicity, we set $d=3$ from the very beginning (since it is physically most relevant case) and to shorten the proof we consider only the case $\alpha \equiv 0$ here. However, the general dimension $d$ and also the case when $\alpha >0$ is treated similarly. We also recall that in what follows the constant~$C$ denotes some universal constant depending only on the data of the problem, i.e. on $\vv_0$, $\theta_0$ and $\f$. If there is any dependence on different quantities, it will be clearly denoted.

\subsection{Galerkin approximations}
We start by considering an orthogonal basis $\{\ww_i\}_{i=1}^{\infty}$ of the space $W_{\mathbf{n},\diver}^{3,2} \hookrightarrow W^{1,\infty}(\Omega)^d$   that is orthonormal in $L_{\mathbf{n},\diver}^2$, see \cite[Appendix A.4]{MNRR} how to construct such basis.
Similarly, let $\{w_j\}_{j=1}^{\infty}$ be a basis of $W^{1,2}(\Omega)$ which is again orthonormal in the space $L^2(\Omega)$.

Moreover, defining the subspaces
\[
\mathbf{W}^n:=\text{span}\{\ww_1,\ldots,\ww_n\}\subset W_{\mathbf{n},\diver}^{3,2}, \qquad W^m:=\text{span}\{w_1,\ldots,w_m\}\subset W^{1,2}(\Omega),
\]
and the associated projections
\[
\mathbf{P}^n(\vv)\equiv \sum_{i=1}^{n} \left(\int_{\Omega} \vv \cdot \ww_i\dx\right)  \ww_i     :W_{\mathbf{n},\diver}^{3,2} \rightarrow \mathbf{W}^n, \qquad P^m(\thet)\equiv \sum_{j=1}^{m} \left(\int_{\Omega} \thet w_j \dx \right) w_j     :W^{1,2}(\Omega) \rightarrow W^m,
\]
we are in position to consider the usual Galerkin approximations.

Now, we construct Galerkin approximations $\{\vv^{n,m}, \thet^{n,m}\}_{n,m=1}^{\infty}$ of the form
\[
\vv^{n,m}(t,x):=\sum_{i=1}^{n} c_i^{n,m}(t) \ww_i(x), \qquad \thet^{n,m}(t,x):=\sum_{j=1}^{m} d_j^{n,m}(t) w_j(x),
\]
where $\mathbf{c}^{n,m}(t):=(c_1^{n,m}(t), \ldots, c_n^{n,m}(t))$ and $\mathbf{d}^{n,m}(t):=(d_1^{n,m}(t), \ldots, d_m^{n,m}(t))$ solve the following system of ordinary differential equations
\begin{equation}\label{eq:vnm}
\int_{\Omega}\partial_t \vv^{n,m} \cdot \ww_i+(\S^{n,m}-\vv^{n,m} \otimes \vv^{n,m} \, g_{k}(|\vv^{n,m}|^2)):\nabla \ww_i\dx =\int_{\Omega}\T_k(\thet_{\ast}^{n,m})\mathbf{f} \cdot  \ww_i \dx,
\end{equation}
for all $i\in\{1,\ldots,n\}$ and
\begin{equation}\label{eq:tnm}
\int_{\Omega}\partial_t \thet^{n,m}\,w_j-(\T_k(\thet^{n,m})\vv^{n,m}+\vc{q}^{n,m})\cdot \nabla w_j\dx=\int_{\Omega}\left(\S^{n,m}:\D(\vv^{n,m})-\T_k(\thet_{\ast}^{n,m}) \vv^{n,m}\cdot\mathbf{f}\right)w_j\dx,
\end{equation}
for all $j\in\{1,\ldots,m\}$. Here, we have used the following abbreviations
\begin{equation}\label{abbr}
\S^{n,m}:=\S^{\star}\left(\thet^{n,m},\D(\vv^{n,m})\right), \qquad \vc{q}^{n,m}:=\vc{q}^*(\thet^{n,m},\nabla \thet^{n,m}), \qquad \thet_{\ast}^{n,m}:=\max\{0,\thet^{n,m}\}.
\end{equation}
In addition, we assume that $\vv^{n,m}$ and $\thet^{n,m}$ satisfy the following initial conditions
\[
\vv_0^{n,m}:=\mathbf{P}^{n}(\vv_0)=\sum_{i=1}^{n}c_{0,i}\ww_i, \qquad \thet_0^{n,m}:=P^m(\thet_0^n)=\sum_{j=1}^{m}d_{0,j}^{n} w_j,
\]
where $\vv_0^{n,m}$ is independent of $m$--parameter and  $\thet_0^{n,m}$ has the following meaning. First, we use the convention that
\begin{equation*}
\eta_0(x):=\left\{\begin{aligned}
&0,  &&x\in \mathbb{R}^d\setminus\Omega,\\
&\ln (\thet_0(x)), &&x\in\Omega.
\end{aligned}\right.
\end{equation*}
Then, we compute the standard regularization of an integrable function $\eta_0$ with the kernel $r_{1/n}$ having the support in a ball of radii~$1/n.$ It means, we define $\eta_0^{n}:=r_{1/n}\ast \eta_0$. Then, since $\thet_0$ and $\ln \thet_0$ are assumed to belong to $L^1(\Omega)$, we have
\begin{subequations}\label{subb}
\begin{align}\label{subb1}
\thet_0^{n}:=e^{\eta_0^n} \xrightarrow[]{n\to \infty} \thet_0 \quad &\text{strongly in } L^1(\Omega),\\
\eta_0^{n} \xrightarrow[]{n\to \infty} \ln \thet_0  \quad &\text{strongly in } L^1(\Omega).\label{subb2}
\end{align}
\end{subequations}
Finally, we apply the projection onto the linear hull of $\{w_j\}_{j=1}^{m}$ to get $P^m(\thet_0^{n}).$ Note that as an immediate consequence of the properties of the projectors we have
\begin{subequations}\label{sub}
\begin{align}\label{sub1}
\vv_0^{n,m}\xrightarrow[]{n\to \infty} \vv_0 \qquad &\text{strongly in } L^2(\Omega)^d,\\
\thet_0^{n,m} \xrightarrow[]{m\to \infty} \thet_0^n \qquad &\text{strongly in } L^2(\Omega).\label{sub2}
\end{align}
\end{subequations}
In addition, since $\log\thet_0\in L^{1}(\Omega)$ by hypothesis, we also have that
\begin{equation}\label{eq:logthet0nm}
\log\thet_0^{n,m}\xrightarrow[]{m \to \infty} \log\thet_0^n \xrightarrow[]{n \to \infty}\log\thet_0 \qquad \text{strongly in } L^1(\Omega).
\end{equation}

We focus on solvability of \eqref{eq:vnm}--\eqref{eq:tnm}. Defining the auxiliary vector-valued functions
\begin{align*}
\mathbf{C}(t)&:=(\mathbf{c}^{n,m}(t),\mathbf{d}^{n,m}(t))=(c_1^{n,m}(t), \ldots, c_n^{n,m}(t), d_1^{n,m}(t), \ldots, d_m^{n,m}(t)),\\
\mathbf{C}_0&:=(\mathbf{c}_{0},\mathbf{d}_0^{n})=(c_{0,1}, \ldots, c_{0,n}, d_{0,1}^{n}, \ldots, d_{0,m}^{n}).
\end{align*}
Then, the system \eqref{eq:vnm}--\eqref{eq:tnm} can be rewritten as
\begin{align*}
\dot{\mathbf{C}}(t)&=\mathcal{F}(t,\mathbf{C}(t)),\\
\mathbf{C}(0)&=\mathbf{C}_0,
\end{align*}
where $\mathcal{F}$ is a Carath\'{e}odory function. Thus, using the classical Carath\'{e}odory theory, see \cite[Chapter 1]{W}, we deduce the existence of solution to \eqref{eq:vnm}--\eqref{eq:tnm} at least for a short time interval. The uniform estimates derived in the next subsection enable us to extend the solution onto the whole time interval~$(0, T)$. %In the next subsections we derive a~priori estimates and we pass to the limit in Galerkin approximation \eqref{eq:vnm}--\eqref{eq:tnm}.
Then, we set $m\to\infty$  and then $n\to\infty$.
Note that some of the estimates are independent of the order of approximation and are frequently used later (after using weak lower semicontinuity of norm in a reflexive space).

\subsection{Estimates independent of  \texorpdfstring{$m$}{m}}

In this part, we start by assuming that $n\in\mathbb{N}$ is arbitrary, but fixed, and we let $m\to \infty$.

\subsubsection{Estimates independent of \texorpdfstring{$m$}{m} for the velocity:}
Multiplying the $i$-th equation in \eqref{eq:vnm} by $c_i^{n,m},$ then taking the sum over $i=1,\ldots, n$ we get
\begin{equation*}
\int_{\Omega} \partial_t \vv^{n,m}\cdot\vv^{n,m}+(\S^{n,m}-\vv^{n,m} \otimes \vv^{n,m} \, g_{k}(|\vv^{n,m}|^2)) :\nabla \vv^{n,m}\dx =\int_{\Omega}\T_k(\thet_{\ast}^{n,m})\mathbf{f} \cdot \vv^{n,m}\dx.
\end{equation*}
Integrating the result over time interval $(0,t)$, we obtain
\begin{equation}\label{eq:auxvnm}
\|\vv^{n,m}(t)\|_{2}^2+ 2 \int_0^t \int_{\Omega}\S^{n,m}:\D(\vv^{n,m})\dx \ds= \|\vv_0^{n,m}\|_{2}^2+2\int_0^t\int_{\Omega} \T_k(\thet_{\ast}^{n,m})\mathbf{f}\cdot  \vv^{n,m}\dx\ds,
\end{equation}
where we have used the identity
\begin{equation}\label{Zero}
\int_{\Omega}(\vv^{n,m} \otimes \vv^{n,m} \, g_{k}(|\vv^{n,m}|^2)):\nabla \vv^{n,m}\dx=0,
\end{equation}
which follows\footnote{The computation goes as follows:
$$
\begin{aligned}
&\int_{\Omega}(\vv^{n,m} \otimes \vv^{n,m} \, g_{k}(|\vv^{n,m}|^2)):\nabla \vv^{n,m}\dx=\frac12 \int_{\Omega}g_{k}(|\vv^{n,m}|^2)\vv^{n,m} \cdot \nabla |\vv^{n,m}|^2\dx\\
&=\frac12 \int_{\Omega}\vv^{n,m} \cdot \nabla G_{k}(|\vv^{n,m}|^2)\dx=-\frac12 \int_{\Omega} G_{k}(|\vv^{n,m}|^2) \diver \vv^{n,m} \dx=0,
\end{aligned}
$$
and $G_k$ denotes the primitive function to $g_k$.} from the fact $\diver \vv^{n,m}=0$ and integration by parts.
%\begin{itemize}
%\item $\int_{\Omega}(\vv^{n,m} \otimes \vv^{n,m} \, g_{k}(|\vv|^2)):\nabla \vv^{n,m}\dx=0$ since $\diver(\vv^{n,m})=0$;
%\item $\int_{\Omega}\S^{n,m}:\nabla \vv^{n,m}\dx=\int_{\Omega}\S^{n,m}:\D(\vv^{n,m})$ by symmetry  of $\S^{\ast}.$
%\end{itemize}
Next, we apply \eqref{nu2}$_1$ to the second term on the left-hand side of~\eqref{eq:auxvnm}. For the right hand side, we use the assumption that~$\f$ is bounded and apply the $L^2-L^2$ H\"{o}lder's inequality  and the definition of $\T_k$, see \eqref{def:Tk} and  the fact that $\|\vv^{n,m}_0\|_2 \le \|\vv_0\|_2$ to obtain for all $t\in (0,T)$
\begin{equation}\label{eq:finalvnm}
\|\vv^{n,m}(t)\|_{2}^2+  \int_0^t \|\D(\vv^{n,m})\|_{p}^p \ds  \leq C(k)\left(1 + \int_0^t \|\vv^{n,m}\|_2^2 \ds\right).
\end{equation}
First, the use of the Gronwall inequality directly leads to the $L^{\infty}-L^2$ estimate for $\vv^{n,m}$. Using this information and the  Korn inequality~\eqref{eq:Korna0} we deduce from \eqref{eq:finalvnm} the following $m$-independent estimate
\begin{equation}\label{eq:prefinalnm}
\sup_{t\in(0,T)}\|\vv^{n,m}(t)\|_{2}^2+  \int_0^T \|\vv^{n,m}\|_{1,p}^p\dt    \leq C(k).
\end{equation}

\subsubsection{Estimates independent of \texorpdfstring{$m$}{m} for the temperature:}
Next, multiplying the $j$-th equation in \eqref{eq:tnm} by $d_j^{n,m}$, taking the sum over $j=1,\ldots, k$ we get
\[
\int_{\Omega}\partial_t \thet^{n,m}\thet^{n,m}-(\T_k(\thet^{n,m})\vv^{n,m}+\vc{q}^{n,m})\cdot\nabla \thet^{n,m}=\int_{\Omega} \left(\S^{n,m}:\D(\vv^{n,m})-\T_k(\thet_{\ast}^{n,m}) (\vv^{n,m}\cdot\mathbf{f})\right)\thet^{n,m}\dx,
\]
and integrating the result over time we arrive at
\begin{equation}
\begin{split}\label{eq:auxthetnm}
&\|\thet^{n,m}(t)\|_2^2+2\int_0^t \int_{\Omega}\kappa(\thet^{n,m})|\nabla \thet^{n,m}|^2 \dx \ds\\
&\qquad = \|\thet_0^{n,m}\|_2^2 +2\int_0^t \int_{\Omega}\left(\S^{n,m}:\D(\vv^{n,m})-\T_k(\thet_{\ast}^{n,m}) (\vv^{n,m}\cdot\mathbf{f})\right)\thet^{n,m}\dx \ds,
\end{split}
\end{equation}
where we have used the abbreviation~\eqref{abbr} and the fact that
\begin{equation}\label{zero2}
\int_{\Omega}\T_k(\thet^{n,m})\vv^{n,m}\cdot \nabla \thet^{n,m}\dx=0,
\end{equation}
which follows\footnote{It is a consequence of the following computation
$$
\begin{aligned}
&\int_{\Omega}\T_k(\thet^{n,m})\vv^{n,m}\cdot \nabla \thet^{n,m}\dx=\int_{\Omega}\vv^{n,m}\cdot \nabla \mathfrak{T}_k(\thet^{n,m})\dx=-\int_{\Omega}\mathfrak{T}_k(\thet^{n,m}) \diver \vv^{n,m}\dx=0,
\end{aligned}
$$
where $\mathfrak{T}_k$ denotes a primitive function to $\T_k$.} from the fact that $\diver(\vv^{n,m})=0$ and integration by parts.
Proceeding as before, we are going to study each term in  \eqref{eq:auxthetnm}. For the left hand side, we use  the lower bound~\eqref{k} on the second term. For the right hand side, we use the fact that $\{\ww_i\}_{i=1}^{\infty}\in W_{\mathbf{n},\diver}^{3,2}\hookrightarrow W^{1,\infty}(\Omega)^d$. Consequently, since the velocity field is for almost all time $t\in (0,T)$ taken from the finite dimensional space $\mathbf{W}^n$, can deduce from \eqref{nu2} that
$$
\left|\S^{n,m}(t,x):\D(\vv^{n,m}(t,x))\right| + |\vv^{n,m}(t,x)| \le C(n) (1+ \|\vv^{n,m}(t)\|^p_2)\le C(k,n),
$$
where the second inequality follows from \eqref{eq:prefinalnm}. Combining both (recall that $\|\thet_0^{n,m}\|_2 \xrightarrow[]{m\to \infty} \|\thet_0^{n}\|_2$) we obtain
\begin{equation}\label{eq:finalthetnm}
\|\thet^{n,m}(t)\|_2^2+\int_0^t \|\nabla\thet^{n,m}\|_2^2 \ds \leq C(n,k)\left(1+  \int_0^t \|\thet^{n,m}\|_2^2 \ds\right)
\end{equation}
and applying the Gronwall lemma and recalling \eqref{eq:prefinalnm},  we have
\begin{equation}\label{eq:finalnm}
\sup_{t\in(0,T)}\left\lbrace\|\vv^{n,m}(t)\|_{2}^2+\|\thet^{n,m}(t)\|_2^2\right\rbrace+  \int_0^T \|\vv^{n,m}\|_{1,p}^p +\|\thet^{n,m}\|_{1,2}^2 \ds   \leq C(n,k).
\end{equation}
Finally, using the three-dimensional version of the interpolation inequality \eqref{interp}, it follows from~\eqref{eq:finalnm} and from the embedding $W^{3,2}\hookrightarrow W^{1,\infty}$ that
\begin{equation}\label{eq:finalnmb}
\int_0^T \|\vv^{n,m}\|_{\frac{10}{3}}^{\frac{10}{3}}+\|\thet^{n,m}\|_{\frac{10}{3}}^{\frac{10}{3}}\dt \le C\int_0^T  \|\thet^{n,m}\|_{2}^{\frac{4}{3}} \|\thet^{n,m}\|_{1,2}^{2}+ \|\vv^{n,m}\|_{2}^{\frac{4}{3}} \|\vv^{n,m}\|_{1,2}^{2} \dt  \le C(n,k).
\end{equation}

%\subsubsection{Estimates for the velocity and temperature independent of $m$:}
%Adding \eqref{eq:finalvnm} and \eqref{eq:finalthetnm} we get
%\begin{multline*}
%\|\vv^{n,m}(t)\|_{2}^2+\|\thet^{n,m}(t)\|_2^2+  \int_0^t \|\vv^{n,m}\|_{W^{1,p}}^p +\|\nabla\thet^{n,m}\|_2^2 \ds
%\leq C(n)\left(1 + \int_0^t \|\thet^{n,m}\|_2^2 + \|\vv^{n,m}\|_2^2 \ds\right),
%\end{multline*}
%and applying the integral form of Gronwall's inequality, we finally get
%\[
%\|\vv^{n,m}(t)\|_{2}^2+\|\thet^{n,m}(t)\|_2^2+  \int_0^t \|\vv^{n,m}\|_{W^{1,p}}^p +\|\nabla\thet^{n,m}\|_2^2 \ds   \leq C(n).
%\]

\subsubsection{Estimates independent of \texorpdfstring{$m$}{m} for time derivatives:}
In order to deduce the compactness of the velocity and the temperature, we also estimate the norms of their time derivative.
More specifically, our goal in this section is to prove that
\begin{equation}\label{eq:finaldtnm}
\int_0^T \|\partial_t \vv^{n,m}\|_2^2 + \|\partial_t \thet^{n,m}\|_{W^{-1,2}(\Omega)}^2 \dt \le C(n,k).
\end{equation}
We start with the velocity field. Since $\{\ww_i\}_{i=1}^{\infty}$ is orthonormal  in $L^2(\Omega)^d$, we have that
\begin{equation}\label{def:vwithc}
\int_0^T\|\partial_t \vv^{n,m}\|^2_{2}\dt =\int_0^T |\dot{\mathbf{c}}^{n,m}(s)|^2 \ds, \qquad \text{and}\qquad \| \vv^{n,m}(t)\|^2_{2}=|\mathbf{c}^{n,m}(t)|^2.
\end{equation}
%
%
%
%\textit{Estimates for the velocity:}
Therefore, multiplying the $i$-equation in \eqref{eq:vnm} by $\dot{c}_i^{n,m}(t),$ summing over $i=1,\ldots,n$ and integrating it over time, we obtain (after using the estimate \eqref{eq:finalnm}) that
%\begin{equation}\label{ptvnm}
%\int_0^T |\dot{\mathbf{c}}^{n,m}(s)|^2 \ds\leq C(n).
%\end{equation}
%\begin{proof}
%In first place, we note that
\begin{equation*}
\int_0^T\|\partial_t \vv^{n,m}\|^2_{2}\dt=\int_0^T \int_{\Omega}\T_k(\thet_{\ast}^{n,m})\mathbf{f} \cdot  \partial_t \vv^{n,m} + \left(\vv^{n,m} \otimes \vv^{n,m}\,g_{k}(|\vv^{n,m}|^2)-\S^{n,m}\right): \nabla\partial_t \vv^{n,m}\dx \dt.
\end{equation*}
Now, we are going to study each term separately. Using the fact that $\{\ww_i\}_{i=1}^{\infty}$ forms a basis of $W_{\mathbf{n},\diver}^{3,2}$, the fact that $\mathbf{W}^n$ is finite dimensional, the~H\"{o}lder inequality and the $\delta$-Young inequality and the a~priori bound \eqref{eq:finalnm}, we have
\begin{align*}
\left|\int_0^T \int_{\Omega}\T_k(\thet_{\ast}^{n,m})\mathbf{f} \cdot \partial_t \vv^{n,m} \dx \ds\right|&\leq \delta \int_0^T |\dot{\mathbf{c}}^{n,m}(s)|^2 \ds + C_\delta(n,k),\\
\left|\int_0^T  (\vv^{n,m} \otimes \vv^{n,m}\,g_{k}(|\vv^{n,m}|^2),\nabla \partial_t \vv^{n,m})\ds \right|&\leq  \delta \int_0^T |\dot{\mathbf{c}}^{n,m}(s)|^2 \ds + C_\delta(n,k).
\end{align*}
Proceeding similarly and using \eqref{nu2}$_2$ we get
\[
\left|\int_0^T  (\S^{n,m},\nabla\partial_t \vv^{n,m})\ds \right|\leq  \delta \int_0^T |\dot{\mathbf{c}}^{n,m}(s)|^2 \ds + C_\delta(n,k)\int_0^T |\mathbf{c}^{n,m}(s)|^{2(p-1)}\ds.
\]
Combining it all, taking $\delta$ small enough and applying \eqref{def:vwithc} and \eqref{eq:finalnm}, we obtain
%\[
%\int_0^T |\dot{\mathbf{c}}^{n,m}(s)|^2 \ds \leq C_\delta(n)\int_0^T\|\thet^{n,m}(s)\|_{L^2(\Omega)} +\| \vv^{n,m}(s)\|_{2}^{2(p-1)}\ds.
%\]
%Using \eqref{eq:finalnm}, we have that right hand side is bounded by a constant independent of $m$.
%\end{proof}
%
\begin{equation}\label{ptvnm}
\int_0^T |\dot{\mathbf{c}}^{n,m}(s)|^2 \ds\leq C(n,k),
\end{equation}
which gives the first part of \eqref{eq:finaldtnm}.

Next, we focus on the estimates for $\partial_t \thet^{n,m}$. Using the orthogonality of the basis of $W^m$ and the regularity of $\partial_t \thet^{n,m}$, we have that for all $t\in (0,T)$
\begin{equation}\label{juhu}
\begin{aligned}
\|\partial_t \thet^{n,m}(t)\|_{W^{-1,2}(\Omega)}&=\sup_{\substack{\varphi\in W^{1,2}(\Omega)\\ \|\varphi\|_{1,2}\le 1}} \langle \partial_t \thet^{n,m}(t), \varphi \rangle_{W^{1,2}(\Omega)} = \sup_{\substack{\varphi\in W^{1,2}(\Omega)\\ \|\varphi\|_{1,2}\le 1}}\int_{\Omega} \partial_t \thet^{n,m}(t,x)  \varphi(x) \dx \\
&=  \sup_{\substack{\varphi\in W^{1,2}(\Omega)\\ \|\varphi\|_{1,2}\le 1}}\int_{\Omega} \partial_t \thet^{n,m}(t,x) P^m( \varphi)(x) \dx.
\end{aligned}
\end{equation}
%
%
%Using the notation $X:=L^2(0,T;W^{-1,2}(\Omega))$, we have that
%\begin{equation}\label{eq:auxinttimethetnm}
%\|\partial_t \thet^{n,m}\|_{X}=\sup_{\substack{\varphi\in X^{\ast}\\ \|\varphi\|_{X^{\ast}}= 1}}\langle \partial_t \thet^{n,m} , \varphi\rangle_{X,X^{\ast}} \leq C(n).
%\end{equation}
%\begin{proof}
%Let $\varphi\in X^{\ast}= L^2(0,T;W^{1,2}(\Omega))$ with $\|\varphi\|_{X^{\ast}}=1$. The above dual-pairing can be written as
%\begin{align*}
%\langle \partial_t \thet^{n,m} , \varphi\rangle_{X,X^{\ast}}&=\int_0^T \langle\partial_t \thet^{n,m},\varphi\rangle_{W^{-1,2}(\Omega),W^{1,2}(\Omega)}\dt=\int_0^T (\partial_t \thet^{n,m},P^m(\varphi))\dt.
%\end{align*}
Now, using \eqref{eq:tnm} we get (we omit writing $(t,x)$)
\begin{multline*}
\int_{\Omega}\partial_t \thet^{n,m} P^m(\varphi) \dx=\int_{\Omega} \left(\S^{n,m}:\D(\vv^{n,m})-\T_k(\thet_\ast^{n,m}) (\vv^{n,m}\cdot\mathbf{e}_d)\right)P^m(\varphi))\dx \\+ \int_{\Omega} \left(\T_k(\thet^{n,m})\vv^{n,m}+\vc{q}^{n,m}\right)\cdot \nabla P^m (\varphi)\dx,
\end{multline*}
and proceeding as before, we obtain
\begin{equation}\label{eq:auxdotthetnm}
\int_{\Omega} \partial_t \thet^{n,m}(t,x) P^m(\varphi)(x)\dx \leq C(n,k) \left(1+\|\vv^{n,m}(t)\|_2^{p}+\|\nabla \thet^{n,m}(t)\|_2 \right)\|P^m(\varphi)\|_{W^{1,2}(\Omega)}.
\end{equation}
Since, the properties of the basis $W^m$ gives that $\|P^m(\varphi)\|_{1,2}\le \|\varphi\|_{1,2}$, we can  use \eqref{eq:auxdotthetnm} in \eqref{juhu}, apply the second power and then integrate over $t\in (0,T)$ to get
\begin{equation*}
\begin{aligned}
\int_0^T \|\partial_t \thet^{n,m}(t)\|^2_{W^{-1,2}(\Omega)}\dt  \leq C(n,k) \int_0^T \left(1+\|\vv^{n,m}(t)\|_2^{2p}+\|\nabla \thet^{n,m}(t)\|^2_2 \right)\dt
\end{aligned}
\end{equation*}
and consequently, using \eqref{eq:finalnm}, we deduce that
%where the function in parentheses on the right-hand side belongs to $L^2(0,T)$ thanks to \eqref{eq:finalnm}.  Furthermore, using the properties of projections  $\|P^m(\varphi)\|_{X^{\ast}} \leq \|\varphi\|_{X^{\ast}}=1$ and integrating in time \eqref{eq:auxdotthetnm} completes the proof.
%\end{proof}
%
%
\begin{equation}\label{eq:auxinttimethetnm}
\int_0^T\|\partial_t \thet^{n,m}\|^2_{W^{-1,2}(\Omega)}\dt \leq C(n,k).
\end{equation}

\subsection{Limit \texorpdfstring{$m\to\infty$}{m} for fixed \texorpdfstring{$n\in \mathbb{N}$}{n}}
In this part, we let $m\to \infty$ but keep $n\in \mathbb{N}$ fixed. Our goal is to identify limits in the equations \eqref{eq:vnm}--\eqref{eq:tnm} as well as in the constitutive relations \eqref{abbr}.

\subsubsection{Weak and strong limits based on a~priori estimates}
Having a~priori uniform estimates \eqref{eq:finalnm}--\eqref{eq:finaldtnm}, we can let $m\to\infty$ and find subsequences $\{c^{n,m}, \thet^{n,m}\}_{m=1}^{\infty}$, that we do not relabel, such that
\begin{subequations}\label{limitm}
\begin{align}
\label{m1}
\mathbf{c}^{n,m} \rightharpoonup^{\ast} \mathbf{c}^n \qquad &\text{weakly}^{\ast} \, \text{in} \, L^{\infty}(0,T), \\
\label{m2}
\dot{\mathbf{c}}^{n,m} \rightharpoonup \dot{\mathbf{c}}^n \qquad &\text{weakly}\phantom{\ast} \, \text{in} \, L^{2}(0,T),\\
\label{m3}
\thet^{n,m} \rightharpoonup^{\ast} \thet^n \qquad &\text{weakly}^{\ast} \, \text{in} \, L^{\infty}(0,T;L^2(\Omega)),\\
\label{m4}
\thet^{n,m} \rightharpoonup \thet^n \qquad &\text{weakly}\phantom{\ast} \, \text{in} \, L^{2}(0,T;W^{1,2}(\Omega)),\\
\label{m5}
\partial_t\thet^{n,m} \rightharpoonup \partial_t\thet^n \qquad &\text{weakly}\phantom{\ast} \, \text{in} \, L^{2}(0,T;W^{-1,2}(\Omega)),\\
\label{m6}
\thet^{n,m} \rightharpoonup \thet^n \qquad &\text{weakly}\phantom{\ast} \, \text{in} \, L^{\frac{10}{3}}(0,T;L^{\frac{10}{3}}(\Omega)).
\end{align}
\end{subequations}
Moreover, using the Aubin--Lions compactness lemma on the sequence $\{\thet^{n,m}\}_{m=1}^{\infty}$, i.e. using \eqref{m4}--\eqref{m6},  and the compact embedding $W^{1,2}(0,T)\hookrightarrow \hookrightarrow \mathcal{C}(0,T)$ on the sequence $\{\mathbf{c}^{n,m}\}_{m=1}^{\infty}$, i.e. using~\eqref{m2}, we have for a subsequence that we do not relabel
\begin{subequations}\label{eq:strongnm}
\begin{align}
\thet^{n,m}\to \thet^n \qquad &\text{strongly in } L^q(0,T;L^q(\Omega)) \text{ for all } q\in[1,10/3),\label{eq:thetnimitm}\\
\mathbf{c}^{n,m} \to \mathbf{c}^{n}  \qquad &\text{strongly in } \mathcal{C}([0,T]).\label{eq:cnmC0}
\end{align}
\end{subequations}
Moreover, it is a simple consequence of our choice of basis and \eqref{eq:cnmC0} that
\begin{equation}\label{eq:vnmtovn}
\vv^{n,m}\rightarrow \vv^n \qquad \text{strongly in } L^{\infty}(0,T;W^{1,\infty}(\Omega)^d),
\end{equation}
and consequently
\begin{equation}\label{eq:DvnmtoDvn}
\D(\vv^{n,m})\rightarrow \D(\vv^n) \qquad \text{strongly in } L^{\infty}(0,T;L^{\infty}(\Omega)^{d\times d}).
\end{equation}
In addition, using \eqref{nu2}, \eqref{eq:thetnimitm} and \eqref{eq:DvnmtoDvn} we can use the Lebesgue dominated convergence theorem to get after denoting $\S^{n,m}:=\S^{\ast}(\thet^{n,m},\D(\vv^{n,m}))$
\begin{equation}\label{eq:Snmlimit}
\S^{n,m}\rightarrow \S^{\ast}(\thet^{n},\D(\vv^{n})):= \S^n \qquad \text{strongly in } L^{q}(0,T;L^{q}(\Omega)^{d\times d})\text{ for all } q\in[1,\infty).
\end{equation}
Similarly, recalling  that $\vc{q}(f):=-\kappa(f)\nabla f$,  combining \eqref{m4} and \eqref{eq:thetnimitm} we get for $\vc{q}^{n,m}:=\vc{q}(\thet^{n,m})$ that
\begin{equation}\label{eq:qnmlimit}
\vc{q}^{n,m} \rightharpoonup \vc{q}^{n}:=\vc{q}(\thet^{n}) \qquad   \text{weakly in } L^2(0,T;L^2(\Omega)^3).
\end{equation}

\subsubsection{Limit in the equations for \texorpdfstring{$\vv^{n,m}$}{v} and \texorpdfstring{$\thet^{n,m}$}{t}}
The convergence results established in \eqref{limitm}--\eqref{eq:qnmlimit} are sufficient to let $m\to \infty$ in~\eqref{eq:vnm} and~\eqref{eq:tnm}. Indeed, we take an arbitrary $\varphi\in \mathcal{C}^{\infty}_0(0,T)$ and multiply the $i$-th equation in~\eqref{eq:vnm} and the  $j$-th equation in~\eqref{eq:tnm} by $\varphi$  and then integrate over time $t \in (0, T)$. Then, using the convergence results \eqref{limitm}--\eqref{eq:qnmlimit} it is easy to pass to the limit in all terms to get the following systems
\begin{equation*}
\int_0^T\left(\int_{\Omega}\partial_t\vv^{n}\cdot \ww_i+\left(\S^{n}-\vv^{n} \otimes \vv^{n} \, g_{k}(|\vv^{n}|^2)\right):\nabla \ww_i-\T_k(\thet_{\ast}^{n})\mathbf{f}\cdot  \ww_i\dx \right)\varphi\dt=0,
\end{equation*}
for all $i=1,\ldots,n$ and
\begin{equation*}
\int_0^T \left(\langle \partial_t\thet^{n}, w_j \rangle -\int_{\Omega}\left(\T_k(\thet^{n})\vv^{n}+\vc{q}^{n}\right) \cdot \nabla w_j+\S^{n}:\D(\vv^{n})w_j-\T_k(\thet_{\ast}^{n}) \vv^{n}\cdot\mathbf{f}w_j \dx \right)\varphi(t)\dt=0,
\end{equation*}
for all $j=1,\ldots,\infty.$ Because $\varphi\in \mathcal{C}^{\infty}_0(0, T)$ can be chosen arbitrarily we can conclude that
\begin{equation}\label{eq:vn}
\int_{\Omega}\partial_t\vv^{n} \cdot\ww_i+\left(\S^{n}-\vv^{n} \otimes \vv^{n} \, g_{k}(|\vv^{n}|^2)\right) :\nabla \ww_i-\T_k(\thet_{\ast}^{n})\mathbf{f} \cdot \ww_i\dx=0,
\end{equation}
for all $i=1,\ldots,n$ and all times $t\in(0,T).$ From the same reason and from the fact that $\{w_j\}_{j=1}^{\infty}$ forms a basis of $W^{1,2}(\Omega)$ we conclude that
\begin{equation}\label{eq:thetn}
\langle \partial_t\thet^{n},\psi\rangle-\int_{\Omega}\left(\T_k(\thet^{n})\vv^{n}+\vc{q}^{n}\right) \cdot \nabla \psi +\S^{n}:\D(\vv^{n})\psi- \T_k(\thet_{\ast}^{n}) \vv^{n}\cdot\mathbf{f}\psi\dx=0,
\end{equation}
is valid for all $\psi\in W^{1,2}(\Omega)$ and for all $t\in(0,T).$

\subsubsection{Attainment of the initial condition for \texorpdfstring{$(\vv^n,\thet^n)$}{}}
We start with the initial condition for the velocity. Since $\vv_0^{n,m}:=\mathbf{P}^{n}(\vv_0)$ we have that $\vv_0^{n,m}\equiv \vv_0^n$ is independent of $m$--parameter. Equivalently, we have $\mathbf{c}^{n,m}_0=\mathbf{c}^n_0,$ for all  $m\in\N$.
Due to \eqref{eq:cnmC0} we have $\mathbf{c}^{n,m}(t) \rightarrow \mathbf{c}^n(t)$ strongly in  $\mathcal{C}([0,T])$ and consequently we get $\mathbf{c}^n(0)=\mathbf{c}^n_0.$
Now, from the definition of $\vv^n(t,x)$ and $\vv^n_0(x)$ it is clear that
\[
\vv^n(0,x)=\vv^n_0(x)
\]
for all $x\in \Omega$.  It remains to show that $\thet^n(0,x)=\thet^n_0(x).$  First, note that a~priori estimates (following from \eqref{m4} and \eqref{m5}) together with the standard parabolic embedding imply that
\[
\thet^{n}\in\left\lbrace z\in L^2(0,T;W^{1,2}(\Omega)), \partial_t z \in L^2(0,T;W^{-1,2}(\Omega)) \right\rbrace   \hookrightarrow \mathcal{C}(0,T;L^2(\Omega)).
\]
Thus, it makes a good sense to define an initial condition. To prove our goal we integrate the equation~\eqref{eq:tnm} over time~$(0,t)$ to get
\begin{align*}
& \int_0^{t}\int_{\Omega}\left(\T_k(\thet^{n,m})\vv^{n,m}+\vc{q}^{n,m}\right) \cdot \nabla w_j+\left(\S^{n,m}:\D(\vv^{n,m})-\T_k(\thet_{\ast}^{n,m}) \vv^{n,m}\cdot\mathbf{f}\right)w_j\dx  \ds \\
&\qquad = \int_{\Omega}\thet^{n,m}(t)w_j\dx-\int_{\Omega}\thet^{n,m}_0 w_j\dx.
\end{align*}
Here, we have used the fact that $\thet^{n,m}_0=\thet^{n,m}(0)$.
Now, using the previous convergence results \eqref{limitm}--\eqref{eq:qnmlimit} and the convergence of the initial condition \eqref{sub2}, we can  let $m\to \infty$ and to obtain for almost all $t\in (0,T)$ that
\[
\int_0^{t}\int_{\Omega}\left(\T_k(\thet^{n})\vv^{n}+\vc{q}^{n}\right) \cdot \nabla w_j+\left(\S^{n}:\D(\vv^{n})-\T_k(\thet_{\ast}^{n}) \vv^{n}\cdot\mathbf{f}\right)w_j \dx \ds = \int_{\Omega}\thet^{n}(t) w_j\dx -\int_{\Omega}\thet^{n}_0 w_j\dx.
\]
Since $\thet^n\in \mathcal{C}(0,T;L^2(\Omega))$, the above identity can be extended for all $t\in (0,T)$. Consequently, letting $t\to 0_+$ in the above identity, we observe that
$$
\int_{\Omega} \thet^{n}(t) w_j \dx \xrightarrow[]{t\to 0_+} \int_{\Omega} \thet^{n}_0 w_j\dx.
$$
But as $\thet^n\in \mathcal{C}(0,T;L^2(\Omega))$ and weak limit as time tends to zero is $\thet^n_0$, we have that
\[
\lim_{t\to 0_+} \|\thet^n(t)-\thet^n_0\|_2=0.
\]

\subsection{Estimates independent of  \texorpdfstring{$n$}{n}}

In this part, we derive estimates that are $n$-independent and that help us to pass to the limit $n\to \infty$. Some of the estimates are also independent of $k$-approximation but if there is dependence on $k$, it will be clearly denoted.

\subsubsection{Nonnegativity for \texorpdfstring{$\thet^n$}{}} First, we show that the temperature $\thet^n$ is nonnegative, i.e.
\begin{equation}\label{minthetl}
\thet^n(t,x)\geq  0 \quad  \text{for a.a. }(t,x)\in (0,T)\times \Omega.
\end{equation}
We consider $\psi(t,x):=\chi_{[0,s]}(t)\min\{0,\thet^n(t,x)\}\leq 0$ as a test function in \eqref{eq:thetn}. Integrating it over time $t\in(0,T)$ we obtain
\[
\int_0^T\langle\partial_t\thet^{n},\psi\rangle_{W^{1,2}(\Omega)} \dt =\int_0^T\int_{\Omega}\left(\T_k(\thet^{n})\vv^{n}+\vc{q}^{n}\right)\cdot \nabla \psi+\S^{n}:\D(\vv^{n})\psi+\T_k(\thet_{\ast}^{n})\vv^{n}\cdot\mathbf{f}\psi \dx \dt.
\]
Next, we show that the right hand side is non-positive. Indeed, using the definition of $\vc{q}^n$, see \eqref{eq:qnmlimit}, and of $\psi$ we have that
$$
\vc{q}^n \cdot \nabla \psi = -\kappa(\thet^n)|\nabla \thet^n|^2 \chi_{[0,s]}\chi_{\{\thet^n\le 0\}}\le 0
$$
almost everywhere in $(0,T)\times \Omega$. Similarly, using \eqref{nu} we can compute
$$
0\le (\S^{n}-\S^*(\thet^n,\mathbf{0})):(\D(\vv^{n})-\mathbf{0}) = \S^{n}:\D(\vv^{n})
$$
and consequently, since $\psi$ is non-positive, we have that
$$
\S^{n}:\D(\vv^{n})\psi \le 0
$$
almost everywhere in $(0,T)\times \Omega$. Further, since $\thet^n_*=\max\{0,\thet^n\}$, it directly follows from the definition of $\psi$ that
$$
\T_k(\thet_{\ast}^{n})\vv^{n}\cdot\mathbf{f}\psi   \equiv 0.
$$
Finally, to estimate the convective term, we introduce the primitive function
%\begin{equation}\label{eq:zeroterms}
%(\T_k(\thet^{n})\vv^{n},\nabla \psi)=0, \qquad (\T_k(\thet_{\ast}^{n}) (\vv^{n}\cdot\mathbf{f}),\psi)=0,
%\end{equation}
%and
%\[
%(\vc{q}^{n},\nabla\psi)\leq 0, \qquad (\S^{n}:\D(\vv^{n}),\psi)\leq 0.
%\]
%\begin{proof}[Proof of \eqref{eq:zeroterms}$_1$]
%Defining the primitive function
\[
\mathfrak{T}_{k}(z):=\int_0^z \T_k(\tilde{z})\mbox{d}\tilde{z},
\]
and we get after  integration by parts (compare with \eqref{zero2}), that
\[
\int_{\Omega} \T_k(\thet^{n})\vv^{n}\cdot \nabla \psi\dx =\int_{\Omega}\vv^n \cdot \nabla \mathfrak{T}_{k}(\psi)\dx=-\int_{\Omega}\diver \vv^n   \mathfrak{T}_{k}(\psi)\dx + \int_{\partial \Omega} \mathfrak{T}_{k}(\varphi)\, \vv^{n}  \cdot\mathbf{n} \,{\rm d}S=0,
\]
where we used the fact that $\diver \vv^n=0$ in $\Omega$  and $\vv^{n}  \cdot\mathbf{n}=0$ on $\partial \Omega$.
%
%
%In the last step, we use the fact that any element of the set $\{x\in\Omega \,|\, \thet^{n}(x)\leq 0\}$  satisfies either that it belongs to the boundary of $\Omega$ or to the null step of $\thet^n$. In either case we have that $\mathcal{P}_{\kappa}(\thet^n)\, \vv^{n}  \cdot\mathbf{n}=0.$
%\end{proof}
%\begin{proof}[Proof of \eqref{eq:zeroterms}$_2$]
%This follows trivialy from definition of $\thet^n_{\ast}=\max\{0,\thet^n\}.$ This is the reason why we need to consider the term $\thet_{\ast}$ in the right-hand side of \eqref{bl}.
%\end{proof}
Hence, we arrive at
\[
0\geq \int_0^T \langle \partial_t\thet^{n},\psi\rangle_{W^{1,2}(\Omega)}\dt =\int_0^{s} \langle \partial_t\psi,\psi\rangle_{W^{1,2}(\Omega)}\dt=\frac{1}{2}\left(\|\psi(s)\|_2^2 - \|\psi(0)\|_2^2 \right).
\]
Finally, the fact that $\psi(x,0)= 0$ a.a. $x\in\Omega$, implies that $\|\psi(s)\|_2=0$ for all $s\in(0,T).$ Then one can easily obtain the desired conclusion \eqref{minthetl}. Consequently, we can now replace $\thet_{\ast}^{n}$ by $\thet^n$ everywhere.

\subsubsection{Renormalization of the temperature equation}
For our result, the very special choice of the test functions in the temperature equation plays an essential role. Therefore, we state already at this point the renormalized version of the temperature equation, which then helps us to get the a~priori estimates and even more,  will be important for proving the entropy (in)equality.  For future use, it is convenient to record a renormalized version of the approximate
equation \eqref{eq:thetn}. Using \eqref{m4} we have that $\thet^n \in L^2(0,T;W^{1,2}(\Omega))$ and moreover we know $\theta^n\ge 0$. Therefore, if we consider $f\in \mathcal{C}^2([0,\infty))$ fulfilling $\|f'\|_{W^{1,\infty}(0,\infty)}<\infty$, we deduce that for arbitrary $\phi \in W^{1,2}(\Omega)\cap L^{\infty}(\Omega)$ the function $\psi:= f'(\thet^n) \phi \in W^{1,2}(\Omega)$ for almost all $t\in (0,T)$. Consequently, such $\varphi$ can be used in \eqref{eq:thetn} and we deduce
\begin{equation}
\begin{split}\label{eq:renorm}
\langle\partial_t f(\thet^{n}),\phi\rangle_{W^{1,2}(\Omega)}-\int_{\Omega}(\vv^n f(\T_k(\thet^n))+ f'(\thet^n)\vc{q}^n)\cdot \nabla\phi - \T_k(\thet^{n}) (\vv^{n}\cdot\mathbf{f})f'(\thet^n)\phi\dx\\
=\int_{\Omega} f''(\thet^n)\vc{q}^n \cdot \nabla\thet^n \phi +\S^{n}:\D(\vv^{n})f'(\thet^n)\phi\dx,
\end{split}
\end{equation}
which is valid for all $\phi\in W^{1,2}(\Omega)\cap L^{\infty}(\Omega)$, any $f\in \mathcal{C}(0,\infty)$ fulfilling $f'\in W^{1,\infty}(0,\infty)$ and for almost all $t\in(0,T).$ Please notice here that for the first term, we used the regularization approach as follows (the computation is done for almost all $t$ and $\thet^n_{\varepsilon}$ denotes the classical regularization with respect to the $t$-variable)
$$
\langle \partial_t \thet^n, \psi \rangle = \lim_{\varepsilon \to 0_+} \langle \partial_t \thet^n_{\varepsilon}, f'(\thet^n_{\varepsilon})\phi \rangle = \lim_{\varepsilon \to 0_+} \int_{\Omega} \partial_t \thet^n_{\varepsilon} f'(\thet^n_{\varepsilon})\phi \dx = \lim_{\varepsilon \to 0_+} \int_{\Omega} \partial_t f(\thet^n_{\varepsilon}) \phi \dx=\langle \partial_t f(\thet^n), \phi \rangle.
$$

\subsubsection{Energy and entropy estimates independent of \texorpdfstring{$n$}{n} and \texorpdfstring{$k$}{k}}

First, we derive the estimates based on the kinetic and internal energy. It is noticeable that they are independent of $n$ and also of $k$. We set $\psi\equiv 1\in W^{1,2}(\Omega)$ in \eqref{eq:thetn} and deduce the identity
\begin{equation}\label{intern}
 \frac{\mathrm{d}}{\dt} \int_\Omega \thet^{n}\dx -\int_\Omega \S^{n}:\D(\vv^{n}) \dx +\int_\Omega  \T_k(\thet^{n}) (\vv^{n}\cdot\mathbf{f})\dx=0,
\end{equation}
which is valid for almost all $t\in (0,T)$. Then, we multiply the $i$-the equation in \eqref{eq:vn} by $c^n_i$, sum over $i=1,\ldots,n$ and we obtain
\begin{equation}\label{eq:auxvn}
\frac{\mathrm{d}}{\dt} \int_\Omega  \frac{|\vv^{n}|^2}{2} \dx +\int_\Omega \S^{n}:\D(\vv^{n}) \dx  -\int_\Omega  \T_k(\thet^{n}) (\vv^{n}\cdot\mathbf{f})\dx=0,
\end{equation}
for a.a. $t\in(0,T).$ Here, we have used the symmetry of $\S^n$, the fact that $\diver \vv^n=0$ and the integration by parts (see the similar computation for the convective term in~\eqref{Zero}). % following facts:
%\begin{itemize}
%	\item \myr{$(\vv^{n} \otimes \vv^{n}\, g_{k}(|\vv^{n}|^2),\nabla \vv^n)=0$ since $\diver \vv^n=0$},
%	\item $(\S^{n},\nabla \vv^n)=(\S^{n},\D( \vv^n))$ by symmetry  of $\S^{\ast}$.
%\end{itemize}
Summing the above identities and using the nonnegativity of the temperature $\thet^n$, we get after integration with respect to time the energy equality
\[
\|\thet^n(t)\|_{1} + \tfrac{1}{2}\|\vv^{n}(t)\|_{2}^2 = \|\thet^n_0\|_{1}+\tfrac{1}{2} \|\vv_0^{n}\|_{2}^2 \le C,
\]
where the second inequality follows from the assumptions on $\thet_0$ and $\vv_0$ and from the properties of the projection $\mathbf{P}^{n}$.
Therefore, we have
\begin{equation}\label{L1thetn}
\|\thet^n\|_{L^{\infty}(0,T;L^1(\Omega))}\le C,
\end{equation}
and
\begin{equation}\label{L2vn}
\|\vv^n\|_{L^{\infty}(0,T;L_{\mathbf{n},\diver}^2)}\le C.
\end{equation}

Next, we show the uniform bound on the entropy. We fix $0<\epsilon\ll 1$, and consider $f(s)=\ln(\epsilon +s)$ and $\phi=1$ in \eqref{eq:renorm}. Note that due to the nonnegativity of the temperature $\thet^n$, such choice of $f$ is admissible. Using the following inequalities
\[
f''(\thet^n)\vc{q}^{n}\cdot \nabla \thet^n = \kappa(\thet^n)\frac{|\nabla\thet^n|^2}{(\epsilon+\thet^n)^2}\ge 0,
\qquad
\S^{n}:\D(\vv^{n})f'(\thet^n)=\frac{\S^{n}:\D(\vv^{n})}{\epsilon+\thet^n}\geq 0,
\]
and moving all terms with positive sign to one side, we observe that for almost all $t\in (0,T)$ we have
\begin{equation*}
-\frac{\mathrm{d}}{\dt}\int_{\Omega}\ln (\epsilon+\thet^n)\dx  + \int_{\Omega}\frac{\S^{n}:\D(\vv^{n})}{\epsilon+\thet^n}+ \frac{\kappa(\thet^n)|\nabla\thet^n|^2}{(\epsilon+\thet^n)^2}\dx \leq \int_{\Omega} \frac{\T_k(\thet^{n})}{\epsilon+\thet^n} (\vv^{n}\cdot\mathbf{f})\dx.
\end{equation*}
Thus, using the fact that $\mathbf{f}$ is bounded, see \eqref{DATA1}, we deduce after using the H\"{o}lder inequality and after integration over time that
\begin{align*}
&\|\ln (\epsilon+\thet^n(t))\|_1+\int_0^t\int_{\Omega}\frac{\S^{n}:\D(\vv^{n})}{\epsilon+\thet^n}+ \frac{\kappa(\thet^n)|\nabla\thet^n|^2}{(\epsilon+\thet^n)^2}\dx \ds \\
&\quad \leq 2\|\max\{0,\ln (\epsilon+\thet^n(t))\}\|_1 + \|\ln (\epsilon+\thet^n_0)\|_1 +C\int_0^t\|\vv^{n}(t)\|_2 \ds.
\end{align*}
Consequently, using the assumptions on $\thet_0^n$, letting $\epsilon \to 0_+$, using the Fatou lemma, the simple algebraic inequality $\max\{0,\ln (\epsilon+\thet^n(t))\}\le C(1+\thet^n(t))$ and the uniform bounds \eqref{L1thetn} and \eqref{L2vn}, we deduce
\begin{equation}\label{eq:logthetn}
\begin{split}
&\sup_{t\in(0,T)} \|\ln \thet^n(t)\|_1+\underline{\kappa}\int_0^T\|\nabla \ln \thet^n\|_2^2\dt\\
&\le \sup_{t\in(0,T)} \|\ln \thet^n(t)\|_1+\int_0^T\int_{\Omega}\frac{\S^{n}:\D(\vv^{n})}{\thet^n}+ \frac{\kappa(\thet^n)|\nabla\thet^n|^2}{(\thet^n)^2}\dx \dt \\
&\quad \le C\left(1+\|\ln \thet_0^n\|_1 +  \sup_{t\in (0,T)} (\|\thet^n(t)\|_1 + \|\vv^n(t)\|^2_2)\right)\le C,
\end{split}
\end{equation}
where $C$ is a uniform constant depending only on data.

\subsubsection{Gradient estimates independent of \texorpdfstring{$n$}{n} possibly depending on \texorpdfstring{$k$}{k}}
Here, we derive the estimates for the velocity and the temperature gradients that are independent of $n$ but may depend on the truncation $k$.
We start with the velocity filed. Integrating in time the equation \eqref{eq:auxvn} and using the estimate~\eqref{L2vn}, we obtain
$$
\int_0^T \int_{\Omega} \S^n : \D(\vv^n)\dx \dt \le C(k).
$$
Thus, the assumption~\eqref{nu2}  and the Korn inequality~\eqref{eq:Korna0} and a~priori estimate~\eqref{L2vn} imply that
\begin{equation}\label{Lpvn}
\int_{0}^T \|\S^n\|_{p'}^{p'} + \|\vv^n\|_{1,p}^p \dt \le C(k).
\end{equation}

Next, we focus on the estimate for the temperature. The starting point is to show that for all $\sigma\in (0,1)$ there holds
\begin{equation}\label{eq:thetnC}
\int_0^T\int_\Omega \frac{|\nabla \thet^n|^2}{(1+\thet^n)^{1+\sigma}}\dx \dt \leq C(\sigma,k).
\end{equation}
%
%\myr{We make use of arguments of Rakotoson. Similar techniques are developed in Boccardo and Gallouet and also in Neumann-Wolf.}
%
%\begin{proof}[Proof of \eqref{eq:thetnC}]
To show \eqref{eq:thetnC},  we define the auxiliary function
$$
f(s):=\frac{(1+s)^{1-\sigma}}{1-\sigma}, \quad f'(s)=(1+s)^{-\sigma}, \quad f''(s)=-\sigma (1+s)^{-1-\sigma},
$$
where $0<\sigma<1$. Then we use \eqref{eq:renorm} with this $f$ and also set $\phi:=1$.
%\[
%\varphi_\sigma(\xi):=\left(1-\frac{1}{(1+\xi)^{\sigma}}\right), \qquad 0\leq \xi <\infty,
%\]
%and their primitive
%\[
%\phi_\sigma(\xi):=\int_0^\xi \varphi_\sigma(z)\mbox{d}z=\left(\xi +\frac{1}{1-\sigma}\left(1-(1+\xi)^{1-\sigma} \right)\right),
%\]
%which satisfies
%\begin{equation}\label{eq:boundsprimitive}
%\left(\frac{\xi}{2}-\frac{2^{\frac{1-\sigma}{\sigma}}}{1-\sigma}\right)\leq \phi_\sigma(\xi)\leq  \xi.
%\end{equation}
%Set $\psi\equiv\varphi_\sigma(\thet^n)\in W^{1,2}(\Omega)$, which is an admissible test function on \eqref{eq:thetn} with
%\[
%\nabla \psi=\frac{\sigma}{(1+\thet^n)^{1+\sigma}}\nabla \thet^n.
%\]
%
%
Doing so, we deduce
\begin{equation*}
\begin{split}
\frac{\mathrm{d}}{\dt} \int_{\Omega} \frac{(1+\thet^n)^{1-\sigma}}{1-\sigma}\dx +\int_{\Omega} \frac{\T_k(\thet^{n}) (\vv^{n}\cdot\mathbf{f})}{(1+\thet^n)^{\sigma}}\dx =\int_{\Omega} \frac{\sigma \kappa(\thet^n)|\nabla \thet^n|^2}{(1+\thet^n)^{1+\sigma}} +\frac{\S^{n}:\D(\vv^{n})}{(1+\thet^n)^{\sigma}}\dx.
\end{split}
\end{equation*}
Here, we used the definition of $\vc{q}^n$ in \eqref{eq:qnmlimit}. Noticing again that
$$
\S^{n}:\D(\vv^{n})\ge 0
$$
almost everywhere in $(0,T)\times \Omega$, we can integrate the above identity over time, use  the H\"{o}lder inequality and the already obtained bounds \eqref{L1thetn} and  \eqref{L2vn} to deduce  \eqref{eq:thetnC}.

Next, we focus on estimates for $\thet^n$ following from \eqref{eq:thetnC}. First, due to nonnegativity of $\thet^n$, it follows form \eqref{eq:thetnC} that for all $\sigma \in (0,1)$
$$
\int_{0}^T \int_{\Omega} \left|\nabla (1+\thet^n)^{\frac{1-\sigma}{2}}\right|^2 \dx \dt \le C(k,\sigma).
$$
Moreover, using also the uniform bound \eqref{L1thetn}, we obtain (recall that $\sigma \in (0,1)$)
$$
\int_{0}^T \int_{\Omega} \left| (1+\thet^n)^{\frac{1-\sigma}{2}}\right|^2 \dx \dt \le C \int_{0}^T \int_{\Omega} (1+\thet^n) \dx \dt \le C.
$$
Consequently, combining the above inequalities and also \eqref{L1thetn}, we have that
$$
\sup_{t\in (0,T)}\|(1+\thet^n(t))^{\frac{1-\sigma}{2}}\|_{2}^2 + \int_{0}^T \|(1+\thet^n)^{\frac{1-\sigma}{2}}\|^2_{1,2}\dt \le C(k,\sigma).
$$
Thus, using the interpolation inequality \eqref{interp} (we use its three dimensional variant only here), we deduce the classical estimate for the temperature of the form
\begin{equation*}%\label{thnLp}
\begin{split}
\int_0^T \int_{\Omega}|(\thet^n)^{1-\sigma}|^{\frac{5}{3}} \dx \dt &\le
\int_0^T \|(1+\thet^n)^{\frac{1-\sigma}{2}}\|_{\frac{10}{3}}^{\frac{10}{3}}\dt \\
&\le \|(1+\thet^n)^{\frac{1-\sigma}{2}}\|_{L^{\infty}(0,T; L^2)}^{\frac{4}{3}} \int_0^T \|(1+\thet^n)^{\frac{1-\sigma}{2}}\|_{1,2}^2\dt \le C(k,\sigma),
\end{split}
\end{equation*}
which is valid for all $\sigma \in (0,1)$. Therefore, it directly follows from the above inequality that
\begin{equation}\label{thnLp}
\begin{split}
\int_0^T \int_{\Omega}|\thet^n|^{q} \dx \dt &\le C(k,q) \quad \textrm{ for all } q\in \left[1,\frac53\right).
\end{split}
\end{equation}
%
%
%It directly follows from \eqref{L1thetn}--\eqref{Lpvn} and assumption \eqref{k} that our goal holds.
%%\end{proof}
%In addition, by an elementary application of H\"{o}lder's inequality, from \eqref{eq:thetnC} it follows that
%\begin{equation}\label{eq:finalthetn}
%\|\thet^n\|_{r(1+\sigma)/(2-r)}+\|\nabla \thet^n\|_{r}\leq C,
%\end{equation}
%with
%\begin{align}
%d&=2, \qquad 1\leq r<2,\quad  0<\sigma<1 \quad  \Longrightarrow\quad  \tfrac{r(1+\sigma)}{2-r}\in (1,\infty),\label{eq:restrictiond2}\\
%d&=3, \qquad 1\leq r<\tfrac{3}{2}, \quad 0<\sigma<\tfrac{3-2r}{3-r}\leq \tfrac{1}{3} \quad  \Longrightarrow \quad  \tfrac{r(1+\sigma)}{2-r}\in\left(\tfrac{3}{2},3\right). \label{eq:restrictiond3}
%\end{align}
%\myr{Similar proof appears on ``On the existence of weak solutions to the equations
%of non-stationary motion of heat-conducting incompressible viscous fluids" and references therein.}
%
%\begin{proof}[Proof of \eqref{eq:finalthetn}]

To derive an estimate on $\nabla \thet^n$, we consider $1\leq r < 5/4.$ Combining the  H\"{o}lder inequality and the previous estimate \eqref{thnLp}, we obtain
\begin{equation*}
\begin{aligned}%\label{eq:gradthetn}
\int_0^T \int_\Omega |\nabla \thet^n|^r \dx \dt & \leq \left(\int_0^T \int_\Omega \frac{|\nabla \thet^n|^2}{(1+\thet^n)^{(1+\sigma)}}\dx \dt \right)^{\tfrac{r}{2}}\left(\int_0^T\int_{\Omega} (1+\thet^n)^{(1+\sigma)\tfrac{r}{2-r}} \dx \dt \right)^{\tfrac{2-r}{2}}\\
&\leq C(k,\sigma,r),
\end{aligned}
\end{equation*}
provided that we can choose $\sigma \in (0,1)$ such that
$$
\frac{r(1+\sigma)}{2-r} < \frac53.
$$
Since $r\in [1,\frac54)$, we can always find $\sigma \in (0,1)$ so that the above inequality holds and therefore
\begin{equation}
\begin{aligned}\label{eq:gradthetn}
\int_0^T \| \thet^n\|_{1,r}^r  \dt & \leq C(k,r) \quad \textrm{ for all } r\in \left[1,\frac54\right).
\end{aligned}
\end{equation}

\subsubsection{Estimates for time derivatives independent of \texorpdfstring{$n$}{n}}
In order to deduce the compactness of the velocity and the temperature, we also need to get a bound on the norms of their time derivatives.
More specifically, our goal in this section will be to prove the uniform bounds in the following spaces
\begin{align}\label{eq:unindt}
\partial_t \vv^{n}\in L^{p'}(0,T;(W_{\mathbf{n},\diver}^{3,2})^{\ast}) \qquad \text{and} \qquad \partial_t \thet^{n}\in L^1(0,T;(W^{1,z}(\Omega))^{\ast}), \qquad \textrm{for all }z>5.
\end{align}

%\textit{Estimates for the velocity:}
We start with the velocity field. Assume that $\bphi\in W_{\mathbf{n},\diver}^{3,2}$ is arbitrary and fulfills $\|\bphi\|_{W_{\mathbf{n},\diver}^{3,2}}\le 1$. We also recall the orthogonality of the basis $\{\vc{w}_i\}_{i=1}^{\infty}$ as well as the continuity of the projection $\mathbf{P}^n$ in the space $W_{\mathbf{n},\diver}^{3,2}$. Then we have due to the regularity of $\partial_t \vv^n$ that
%\begin{proof}
%Let $\varphi\in Y$ with $\|\varphi\|_{Y}=1$. Note that the above dual-pairing can be written as
\begin{align*}
\langle \partial_t \vv^{n}(t) , \bphi\rangle_{W_{\mathbf{n},\diver}^{3,2}}&= \int_{\Omega} \partial_t \vv^{n}\cdot \bphi\dx =\int_{\Omega}\partial_t \vv^{n} \cdot \mathbf{P}^n(\bphi)\dx.
\end{align*}
Now, using \eqref{eq:vn} we get
\[
\int_{\Omega}\partial_t\vv^{n}\cdot \mathbf{P}^n(\bphi)\dx = \int_{\Omega} \left(\vv^{n} \otimes \vv^{n} \, g_{k}(|\vv^{n}|^2)-\S^{n} \right) : \nabla \mathbf{P}^n(\bphi)+\T_k(\thet^{n})\mathbf{f}\cdot \mathbf{P}^n(\bphi)\dx,
\]
and proceeding as before and using the growth assumption \eqref{nu2},  we obtain that for all $t\in (0,T)$
\begin{align*}
\langle \partial_t\vv^{n}(t),\mathbf{P}^n(\bphi)\rangle_{W_{\mathbf{n},\diver}^{3,2}}&\leq C(k) \left(1+\|\vv^n(t)\|_{1,p}^{p-1}\right)\|\mathbf{P}^n(\bphi)\|_{W_{\mathbf{n},\diver}^{3,2}}\\
&\le C(k) \left(1+\|\vv^n(t)\|_{1,p}^{p-1}\right)\|\bphi\|_{W_{\mathbf{n},\diver}^{3,2}}\leq C(k) \left(1+\|\vv^n(t)\|_{1,p}^{p-1}\right).
\end{align*}
Consequently, we have
\begin{equation*}
\|\partial_t\vv^{n}(t)\|_{(W_{\mathbf{n},\diver}^{3,2})^{\ast}}\leq C(k) \left(1+\|\vv^n(t)\|_{1,p}^{p-1}\right)
\end{equation*}
and raising this inequality to the power $p'$ and integrating the result over $(0,T)$, using also the already obtained bound \eqref{Lpvn}, we deduce
%Using the notation $Y:=L^{p}(0,T;W_{\mathbf{n},\diver}^{3,2})$, we are going to prove that
\begin{equation}\label{eq:timevn}
\int_0^T \|\partial_t \vv^{n}\|_{(W_{\mathbf{n},\diver}^{3,2})^{\ast}}^{p'}\dt \le C(k)\int_0^T 1+\|\vv^n(t)\|_{1,p}^{p} \dt \leq C(k).
\end{equation}

%
%
%where the function in parentheses on the right-hand side belongs to $L^{p'}(0,T)$ thanks to \eqref{Lpvn}.
%Finally, we just need to apply the properties of projections  $\|\mathbf{P}^n(\varphi)\|_{Y} \leq \|\varphi\|_{Y}=1$ to complete the proof.
%%\end{proof}
%

%\textit{Estimates for the temperature:}
Finally, we focus on the estimate for time derivative of $\thet^n$. Recall that
\begin{equation}\label{eq:dotthetn}
\langle \partial_t\thet^{n},\psi\rangle-\int_{\Omega}(\T_k(\thet^{n})\vv^{n}+\vc{q}^{n})\cdot \nabla \psi\dx  +\S^{n}:\D(\vv^{n})\psi-\T_k(\thet^{n}) (\vv^{n}\cdot\mathbf{f})\psi\dx=0,
\end{equation}
is valid for all $\psi\in W^{1,2}(\Omega)$ and for a.a. $t\in(0,T).$ Let us consider $\psi \in W^{1,z}(\Omega)$ with $z>5$ fulfilling $\|\psi\|_{1,z} \le 1$. Then using the fact that $W^{1,z}(\Omega)\hookrightarrow L^{\infty}(\Omega)$, we get by using the H\"{o}lder inequality that for almost all $t\in (0,T)$ the following inequality holds
\begin{equation*}
\begin{split}
\langle \partial_t\thet^{n}(t),\psi\rangle&=\int_{\Omega}(\T_k(\thet^{n}(t))\vv^{n}(t)+\vc{q}^{n}(t))\cdot \nabla \psi  +\S^{n}(t):\D(\vv^{n}(t))\psi-\T_k(\thet^{n}(t)) (\vv^{n}(t)\cdot\mathbf{f}(t))\psi\dx\\
&\le C(k)(\|\vv^n(t)\|_{z'}+\|\nabla \thet^n(t)\|_{z'})\|\psi\|_{1,z}+ C(k)(\|\S^{n}(t):\D(\vv^{n}(t)\|_1+\|\vv^{n}(t)\|_1)\|\psi\|_{\infty}.
\end{split}
\end{equation*}
Hence, using \eqref{nu2}, and the fact that $z>2$, we deduce that for almost all $t\in (0,T)$ there holds
\begin{equation*}
\begin{split}
\|\partial_t \thet^n(t)\|_{(W^{1,z}(\Omega))^*}&=\sup_{\varphi\in W^{1,z}(\Omega), \|\varphi\|_{1,z}\le 1} \langle \partial_t\thet^{n}(t),\psi\rangle\\
&\le  C(k)\left(1+\|\vv^n(t)\|_{2}+\|\nabla \thet^n(t)\|_{z'}+\|\vv^{n}(t)\|_{1,p}^p\right).
\end{split}
\end{equation*}
Since $z>5$ we have that $z'<\frac54$ and therefore integration of the above inequality over $(0,T)$ and applying the uniform bounds \eqref{L2vn}, \eqref{Lpvn} and  \eqref{eq:gradthetn} leads to
\begin{equation}\label{eq:thetantime}
\begin{split}
\int_0^T \|\partial_t \thet^n\|_{(W^{1,z}(\Omega))^*}\dt \le C(k,z) \quad \textrm{ for all }z>5.
\end{split}
\end{equation}

\subsection{Identification of limits as \texorpdfstring{$n\to\infty$}{n}}
In this final part we let $n\to \infty$ and complete the proof of Lemma~\ref{mainlemma}.
\subsubsection{Convergence results based on a~priori estimates}
Having a~priori uniform estimates \eqref{L2vn}, \eqref{Lpvn}, \eqref{eq:gradthetn}, \eqref{thnLp} and \eqref{eq:timevn}, using the definition of $\vc{q}^{n}$ and the assumption \eqref{k},  and using the reflexivity of underlying spaces (or separability of their pre-dual spaces), we can let $n\to\infty$ and find subsequences $\{\vv^{n}, \thet^{n}\}_{n=1}^{\infty}$ (that are not relabeled) such that
\begin{subequations}\label{limitl}
\begin{align}
\label{l1}
\vv^{n} &\rightharpoonup^{\ast} \vv &&\text{weakly$^{\ast}$  in }  L^{\infty}(0,T;L^2(\Omega)^d), \\
\label{l2}
\vv^{n} &\rightharpoonup \vv &&\text{weakly$^{\phantom{\ast}}$ in }  L^{p}(0,T;W_{\mathbf{n},\diver}^{1,p}(\Omega)), \\
\label{l3}
\partial_t\vv^n &\rightharpoonup \partial_t\vv &&\text{weakly$^{\phantom{\ast}}$ in } L^{p'}(0,T;(W_{\mathbf{n},\diver}^{3,2})^{\ast}),\\
\label{l4}
\thet^{n} &\rightharpoonup \thet &&\text{weakly$^{\phantom{\ast}}$ in }  L^{q}(0,T;L^{q}(\Omega)), \text{ for all } q\in [1,5/3),\\
\label{l5}
\thet^{n} &\rightharpoonup \thet &&\text{weakly$^{\phantom{\ast}}$ in } L^{r}(0,T;W^{1,r}(\Omega)), \text{ for all } r\in [1,5/4), \\
%\label{l6}
%\log\thet^{n} \rightharpoonup^{\ast} \log\thet \qquad &\text{weakly}^{\ast} \, \text{in} \, L^{\infty}(0,T;L^1(\Omega))\\
\label{limitSbar}
\S^{n} &\rightharpoonup \S &&\text{weakly$^{\phantom{\ast}}$ in } L^{p'}(0,T;L^{p'}(\Omega)^{3\times 3}),\\
\label{limitqbar}
\vc{q}^{n} &\rightharpoonup \vc{q} &&\text{weakly$^{\phantom{\ast}}$ in }  L^{r}(0,T;L^{r}(\Omega)^d), \text{ for all }r\in [1,5/4).
\end{align}
\end{subequations}
Moreover, using the generalized version of the Aubin--Lions compactness lemma, the estimate \eqref{eq:thetantime} and the convergence results \eqref{l2} and \eqref{l5}, we obtain
\begin{subequations}\label{kok}
\begin{align}\label{eq:strongthetn}
\thet^{n}&\to \thet &&\text{strongly in } L^1(0,T;L^1(\Omega)),\\
\label{eq:stongvvn}
\vv^{n}&\to \vv &&\text{strongly in } L^p(0,T;L_{\mathbf{n},\diver}^p).
\end{align}
\end{subequations}
Then, going back to the uniform bounds \eqref{eq:logthetn} and \eqref{eq:thetnC}, we also have for a proper subsequence
\begin{subequations}\label{limitlb}
\begin{align}
%\label{l1b}
%\vv^{n} \rightharpoonup^{\ast} \vv \qquad &\text{weakly}^{\ast} \, \text{in} \, L^{\infty}(0,T;L^2(\Omega)^d) \\
%\label{l2b}
%\vv^{n} \rightharpoonup \vv \qquad &\text{weakly}\phantom{\ast}  \, \text{in} \, L^{p}(0,T;W_{\mathbf{n},\diver}^{1,p}(\Omega)) \\
%\label{l3b}
%\partial_t\vv^n \rightharpoonup \partial_t\vv \qquad &\text{weakly}\phantom{\ast} \, \text{in} \, L^{p'}(0,T;(W_{\mathbf{n},\diver}^{3,2})^{\ast})\\
\label{l3b}
\thet^{n} &\to \thet &&\text{strongly in } L^{q}(0,T;L^{q}(\Omega)) \textrm{ for all }q\in [1,5/3),\\
\label{l4b}
(1+\thet^{n})^{\frac{1-\sigma}{2}} &\rightharpoonup^{\ast}  (1+\thet)^{\frac{1-\sigma}{2}} &&\text{weakly$^{\ast}$ in } L^{\infty}(0,T;L^{2}(\Omega)) \textrm{ for all }\sigma\in [0,1), \\
\label{l5b}
(1+\thet^{n})^{\frac{1-\sigma}{2}} &\rightharpoonup  (1+\thet)^{\frac{1-\sigma}{2}} &&\text{weakly$^{\phantom{\ast}}$ in } L^{2}(0,T;W^{1,2}(\Omega)) \textrm{ for all }\sigma\in (0,1), \\
\label{l6b}
\ln\thet^{n} &\rightharpoonup \ln \thet &&\text{weakly$^{\phantom{\ast}}$ in } L^{2}(0,T;W^{1,2}(\Omega)).
%\label{l7b}
%\partial_t\thet^{n} \rightharpoonup \partial_t\thet \qquad &\text{weakly}\phantom{\ast} \, \text{in} \, L^1(0,T;W^{-1,r_2}(\Omega)),
\end{align}
\end{subequations}
In addition, using the Fatou lemma, the convergence result \eqref{eq:strongthetn} and the estimate \eqref{eq:logthetn}, we have
\begin{equation}
\ln \thet \in L^{\infty}(0,T; L^1(\Omega)).\label{thetlogg}
\end{equation}

\subsubsection{Limit for the velocity equation}
Having the convergence results \eqref{limitl} and \eqref{kok}, we can easily pass to the limit $n\to \infty$ in \eqref{eq:vn} and to conclude for all $\ww\in W_{\mathbf{n},\diver}^{3,2}$ and a.a. $t\in(0,T)$ that
\begin{equation}\label{eq:finalvdiv}
\langle \partial_t\vv,\ww\rangle +\int_{\Omega}(\S-\vv \otimes \vv \, g_{k}(|\vv|^2)):\nabla \ww-\T_k(\thet)\mathbf{f}\cdot\ww\dx=0.
\end{equation}
In addition, as the space $W_{\mathbf{n},\diver}^{3,2}(\Omega)$ is dense in $W_{\mathbf{n},\diver}^{1,p}$ for all $p\geq 1$ we have that \eqref{eq:finalvdiv} holds true for all $\ww\in W_{\mathbf{n},\diver}^{1,p}$ and almost all $t\in(0,T)$ and also that
$$
\partial_t\vv \in L^{p'}(0,T; (W_{\mathbf{n},\diver}^{1,p})^{\ast}).
$$
Consequently, the standard parabolic interpolation gives that
$$
\vv\in \mathcal{C}([0,T]; L^2(\Omega)^3).
$$
To complete the part of the proof of Lemma~\ref{mainlemma}, which is related to the velocity field, we need to identify $\S$ and also the initial condition for $\vv(0)$.
\subsubsection{Attainment of the initial condition for \texorpdfstring{$\vv$}{v}}
The arguments concerning the attainment of the initial conditions $\vv_0$ are standard.
For sake of completeness, we included all details below. Since $\vv\in \mathcal{C}(0,T;L^2(\Omega)^3),$ we get that
\[
\vv(t)\rightarrow \vv(0) \quad \text{strongly in } L_{\mathbf{n},\diver}^2 \text{ as } t\to 0_+.
\]
In what follows, we show that
\[
\vv(t)\rightharpoonup \vv_0 \quad \text{weakly in } L_{\mathbf{n},\diver}^2 \text{ as } t\to 0_+,
\]
and these convergence results together (due to the uniqueness of weak convergence) identify the limit \eqref{limt0approx}$_1$, that we want to prove.

Let $0<\varepsilon\ll 1$ and $t\in (0,T-\varepsilon)$ be arbitrary. We consider the auxiliary function $\eta$ defined in \eqref{def:eta}, multiply \eqref{eq:vn} by this $\eta$, and integrate the result over $(0, T)$ to obtain for every $i=1,\ldots,n$
\[
\int_0^T \int_{\Omega}\partial_t\vv^{n}\cdot \ww_i \eta+ (\S^{n}-\vv^{n} \otimes \vv^{n} \, g_{k}(|\vv^{n}|^2)):\nabla \ww_i\eta -\T_k(\thet^{n})\mathbf{f}\cdot \ww_i\eta \, \dx \ds =0.
\]
Next, we integrate by parts in the first term, use that $\eta(T) = 0$, and the equality $\vv^n(0)=\mathbf{P}^{n}(\vv_0)$ to get
\[
\int_0^T \int_{\Omega}-\vv^{n}\cdot \ww_i \eta'+ (\S^{n}-\vv^{n} \otimes \vv^{n} \, g_{k}(|\vv^{n}|^2)):\nabla \ww_i\eta -\T_k(\thet^{n})\mathbf{f}\cdot \ww_i\eta \, \dx \ds=\int_{\Omega}\mathbf{P}^{n}(\vv_0)\cdot \ww_i\eta(0)\dx.
\]
This identity  is ready for the use of the convergence results \eqref{limitl}--\eqref{kok} as well as
the convergence of the projection $\mathbf{P}^{n}(\vv_0)$ to obtain for any $i\in\mathbb{N}$ that
\[
\int_0^T \int_{\Omega}-\vv\cdot \ww_i \eta'+ (\S-\vv \otimes \vv \, g_{k}(|\vv|^2)):\nabla \ww_i\eta -\T_k(\thet)\mathbf{f}\cdot \ww_i\eta \, \dx \ds=\int_{\Omega}\vv_0\cdot \ww_i\eta(0)\dx.
\]
Using the definition of $\eta$, i.e. using that $\eta(\tau) = 1$ for $\tau\in[0, t)$ and $\eta(\tau) = 0$ for $\tau\in(t + \varepsilon, T]$, and $\eta'(\tau )=-\frac{1}{\varepsilon}$ for $\tau \in (t, t + \varepsilon)$, we obtain
\[
\frac{1}{\varepsilon}\int_t^{t+\varepsilon} \int_{\Omega}\vv\cdot \ww_i \dx\ds + \int_0^{t+\varepsilon}\int_{\Omega}(\S-\vv \otimes \vv \, g_{k}(|\vv|^2)):\nabla \ww_i\eta -\T_k(\thet)\mathbf{f}\cdot \ww_i\eta \, \dx \ds=\int_{\Omega}\vv_0\cdot \ww_i\dx.
\]
Next, since $\vv\in \mathcal{C}(0,T;L_{\mathbf{n},\diver}^2)$, we can easily let $\varepsilon \to 0_+$ to get
\[
\int_{\Omega}\vv(t)\cdot \ww_i \dx + \int_0^{t}\int_{\Omega}(\S-\vv \otimes \vv \, g_{k}(|\vv|^2)):\nabla \ww_i\eta -\T_k(\thet)\mathbf{f}\cdot \ww_i\eta \, \dx \ds=\int_{\Omega}\vv_0\cdot \ww_i\dx
\]
for all $t\in (0,T)$. Therefore, we see that
\[
\lim_{t\to 0_+}\int_{\Omega} \vv(t)\cdot \ww_i\dx =\int_{\Omega}\vv_0\cdot \ww_i\dx.
\]
This holds for every $i \in\mathbb{N}$, and since $\{\ww_i\}_{i=1}^{\infty}$  is a basis of the space $W_{\mathbf{n},\diver}^{3,2}$, this is nothing more than
\[
\vv(t)\rightharpoonup \vv_0 \quad \text{weakly in } (W_{\mathbf{n},\diver}^{3,2})^* \text{ as } t\to 0_+.
\]
Finally, since $W_{\mathbf{n},\diver}^{3,2}$ is dense in $L_{\mathbf{n},\diver}^2$, the weak convergence result that we expected is also true, and it identifies the strong limit of the initial condition in $L_{\mathbf{n},\diver}^2$ as required in  \eqref{limt0approx}$_1$.

\subsubsection{Identification of \texorpdfstring{$\S$}{S}}
Next, we show that
\begin{equation}\label{minty}
\S = \S^*(\thet, \D(\vv)) \qquad \text{a.e. in } (0,T)\times \Omega.
\end{equation}
%\begin{proof}[Proof of \eqref{minty}]
First, we notice that we can extend the solution to the interval $(0,T+1)$, e.g., by extending $\mathbf{f}$ by zero outside of $(0,T)$.
Multiplying the $i$-th equation in  \eqref{eq:vn} by $c_i^n$, summing the result over $i=1,\ldots, n$ and integrate over $(0,T+s)$, where $s\in (0,1)$, we deduce the following identity (the convective term vanishes)
\begin{multline*}
\int_0^{T+s} \int_\Omega \S^n :\D(\vv^n)\dx\dt= -\int_0^{T+s}\int_{\Omega} \partial_t \vv^n\cdot \vv^n+\T_k(\thet^n)\mathbf{f}\cdot \vv^n\dx\dt\\
=-\frac{\|\vv^n(T+s)\|_2^2}{2} + \frac{\|\mathbf{P}^n(\vv_0)\|_2^2}{2} - \int_0^{T+s}\int_{\Omega}\T_k(\thet^n)\mathbf{f}\cdot \vv^n\dx\dt.
\end{multline*}
Using the fact that $\S^n :\D(\vv^n)\ge 0$, we can integrate the above identity over $s\in (0,\varepsilon)$ to get
\begin{multline*}
\int_0^{T} \int_\Omega \S^n :\D(\vv^n)\dx\dt\le -\frac{1}{\varepsilon}\int_T^{T+\varepsilon}\frac{\|\vv^n(t)\|_2^2}{2} \dt + \frac{\|\mathbf{P}^n(\vv_0)\|_2^2}{2} - \frac{1}{\varepsilon}\int_0^{\varepsilon}\int_0^{T+s}\int_{\Omega}\T_k(\thet^n)\mathbf{f}\cdot \vv^n\dx\dt \ds.
\end{multline*}
Then, we may directly use the convergence results \eqref{limitl} and \eqref{kok}, together with the Fatou lemma to get
\begin{multline*}
\limsup_{n\to \infty}\int_0^{T} \int_\Omega \S^n :\D(\vv^n)\dx\dt\le -\frac{1}{\varepsilon}\int_T^{T+\varepsilon}\frac{\|\vv(t)\|_2^2}{2} \dt + \frac{\|\vv_0\|_2^2}{2} - \frac{1}{\varepsilon}\int_0^{\varepsilon}\int_0^{T+s}\int_{\Omega}\T_k(\thet)\mathbf{f}\cdot \vv\dx\dt \ds.
\end{multline*}
Finally, letting $\varepsilon \to 0_+$, we get (using that $\vv\in \mathcal{C}([0,T]; L^2(\Omega))$)
\begin{equation}\label{SS1}
\limsup_{n\to \infty}\int_0^{T} \int_\Omega \S^n :\D(\vv^n)\dx\dt\le -\frac{\|\vv(T)\|_2^2}{2} \dt + \frac{\|\vv_0\|_2^2}{2} - \int_0^{T}\int_{\Omega}\T_k(\thet)\mathbf{f}\cdot \vv\dx\dt.
\end{equation}
Then, setting $\ww:=\vv$ in \eqref{eq:finalvdiv}, integrating over $(0,T)$ and using the fact that $\vv(0)=\vv_0$, we get the identity
\begin{equation}\label{SS2}
\int_0^{T} \int_\Omega \S :\D(\vv)\dx\dt= -\frac{\|\vv(T)\|_2^2}{2} \dt + \frac{\|\vv_0\|_2^2}{2} - \int_0^{T}\int_{\Omega}\T_k(\thet)\mathbf{f}\cdot \vv\dx\dt.
\end{equation}
Comparing \eqref{SS1} and \eqref{SS2}, we directly derive
\begin{equation}
\limsup_{n\to \infty}\int_0^T \int_\Omega \S^n :\D(\vv^n)\dx\dt \leq \int_0^{T} \int_\Omega \S :\D(\vv)\dx\dt.
\end{equation}
%where the last identity follows from \eqref{eq:finalvSbar} with  $\ww\equiv \vv.$
Consequently, using this inequality, the monotonicity  and the growth assumption \eqref{nu1}--\eqref{nu2}, the strong convergence result \eqref{eq:strongthetn}, the weak convergence results \eqref{l2}, \eqref{limitSbar} and the Lebesgue dominated convergence theorem, we deduce that for all $\overline{\D}\in L^p(Q,\mathbb{R}^{d\times d})$ there holds
\begin{equation}\label{mintytoL1}
\begin{split}
0&\leq \limsup_{n\to\infty}\int_0^T \int_\Omega (\S^n-\S^{\ast}(\thet^n,\overline{\D})):(\D(\vv^n)-\overline{\D})\dx\dt\\
&\leq \int_0^T \int_\Omega (\S-\S^{\ast}(\thet,\overline{\D})):(\D(\vv)-\overline{\D})\dx\dt.
\end{split}
\end{equation}
The choice $\overline{\D}:=\D(\vv)\pm\lambda \widetilde{\D}$ with $\lambda>0$ then leads (after division by $\lambda$ and letting $\lambda \to 0_+$) to
\begin{equation*}%\label{mintytoL1}
\begin{split}
0=\int_0^T \int_\Omega (\S-\S^{\ast}(\thet,\D(\vv))):\widetilde{\D} \dx\dt
\end{split}
\end{equation*}
for all $\overline{\D}\in L^p(Q,\mathbb{R}^{d\times d})$. This directly implies \eqref{minty}.
%\end{proof}

In addition, we show that
\begin{equation}\label{SDvL1limit}
\S^n:\D(\vv^n) \rightharpoonup \S:\D(\vv) \qquad \text{weakly in } L^1(0,T;L^1(\Omega)^{3\times 3}).
\end{equation}
%\begin{proof}[Proof of \eqref{SDvL1limit}]
Indeed, we set $\overline{\D}\equiv \D(\vv)$ in \eqref{mintytoL1} and we get
\[
\limsup_{n\to\infty}\int_0^T \int_\Omega (\S^n-\S^{\ast}(\thet^n,\D(\vv))):(\D(\vv^n)-\D(\vv))\dx\dt=0.
\]
However, thanks to \eqref{nu1}, this implies
\begin{equation}\label{auxSDvL1limit}
(\S^n-\S^{\ast}(\thet^n,\D(\vv))):(\D(\vv^n)-\D(\vv))\to 0 \qquad \text{strongly in } L^1((0,T)\times \Omega).
\end{equation}
Since
\[
\S^{\ast}(\thet^n,\D(\vv)):(\D(\vv^n)-\D(\vv))\rightharpoonup 0 \qquad \text{weakly in } L^1((0,T)\times \Omega),
\]
which follows from \eqref{l2}, \eqref{eq:strongthetn} and \eqref{nu2}, we see that \eqref{auxSDvL1limit} directly implies \eqref{SDvL1limit}.
%\end{proof}

\subsubsection{Limit in the temperature gradient and the renormalized temperature equation}
First, we let~$n\to \infty$ in \eqref{eq:thetn}. Using the weak convergence results \eqref{l2}--\eqref{l5}, the strong convergence results \eqref{kok}, the essential convergence stated in \eqref{SDvL1limit}, the convergence of the initial condition \eqref{subb1} and using also the integration by parts with respect to time variable, we deduce that for all $\psi \in \mathcal{C}_0^1(-\infty, T; \mathcal{C}^1(\overline{\Omega}))$ there holds
\begin{equation}\label{eq:thetA}
-\int_0^T \int_{\Omega} \thet  \partial_t\psi+\left(\T_k(\thet)\vv+\vc{q}\right) \cdot \nabla \psi +\S:\D(\vv)\psi+ \T_k(\thet) \vv\cdot\mathbf{f}\psi\dx \dt=\int_{\Omega} \thet_0 \psi(0)\dx.
\end{equation}
In addition, having the weak compactness stated in \eqref{auxSDvL1limit}, we can follow step by step the procedure developed in \cite{Pr97}, see also \cite{AbBuKa19}, to pass to the limit in \eqref{eq:renorm} and to obtain \eqref{temp}.
%we conclude that
%\begin{multline}\label{eq:renormlimit}
%(\partial_t f(\thet),\psi)-(\vv f(\T_k(\thet)),\nabla\psi)-(\vc{q} f'(\thet),\nabla\psi)-(\vc{q} f''(\thet)\nabla\thet, \psi)\\
%-(\S:\D(\vv), f'(\thet)\psi)+(\T_k(\thet) (\vv\cdot\mathbf{f}), f'(\thet)\psi)=0,
%\end{multline}
%is valid for all $\psi\in W^{1,d^+}(\Omega)$   and for a.a. $t\in(0,T).$
%
%Passing to the limit in all terms is standard and we omit the details. In fact, by means of the density of  smooth functions we conclude that the same holds also for for all $\psi\in \left( W^{1,2}\cap L^{\infty}\right)(\Omega)$.
Finally, to obtain \eqref{entp}, we set $f'(s):=(s-m)_+/s$ and $\phi\equiv 1$ in \eqref{eq:renorm} to get after integration over $(0,T)$ that
\begin{equation*}
\begin{split}
\int_{Q}\chi_{\{\thet^n>m\}}\frac{m\kappa(\thet^n)|\nabla \thet^n|^2}{(\thet^n)^2}\dx \dt \le \int_{Q} \frac{\S^n:\D\vv^n (\thet^n-m)_+}{\thet^n}+ \frac{\T_k(\thet^{n}) (\vv^{n}\cdot\mathbf{f})(\thet^n-m)_+}{\thet^n} \dx \dt.
\end{split}
\end{equation*}
Using the weak lower semicontinuity and also all the above established convergence results, we may let $n\to \infty$ to get
\begin{equation}
\begin{split}\label{eq:renormnn}
\int_{Q}\chi_{\{\thet>m\}}\frac{m\kappa(\thet)|\nabla \thet|^2}{(\thet)^2}\dx \dt \le \int_{Q} \frac{\S:\D\vv (\thet-m)_+}{\thet}+ \frac{\T_k(\thet) (\vv\cdot\mathbf{f})(\thet-m)_+}{\thet} \dx \dt.
\end{split}
\end{equation}
Thus, we see that taking $\limsup$ as $m\to \infty$, one conclude~\eqref{entp}.

\subsubsection{Attainment of the initial condition for \texorpdfstring{$\thet$}{t}}
This is a very classical part and we refer the reader e.g. to \cite{BuMaRa09} or \cite{BoGa96,BoDaGaOr97}.

\subsubsection{On the pressure} In this final subsection, we sketch the proof the existence of the pressure $\pi\in  L^{p'}(Q)$. We refer for details to \cite{BuMR} or \cite{BuMaRa09}. By the Helmholtz decomposition we observe that
\[
W_{\mathbf{n}}^{1,p}:=W_{\mathbf{n},\diver}^{1,p}\oplus \{\nabla \varphi: \varphi\in W^{2,p}(\Omega), \nabla\varphi\cdot\mathbf{n}=0 \text{ on }\partial\Omega \}.
\]
Having $(\vv,\thet)$ we introduce $\pi$ as the solution of the following problem
\begin{equation}\label{eq:shortpi}
\pi:=(-\Delta_N)^{-1}\diver\left(\diver\S+\diver(\vv\otimes\vv \, g_k(|\vv|^2))-\T_k(\thet)\mathbf{f}\right),
\end{equation}
where $(-\Delta_N)$ denotes the Laplace operator together with homogeneous Neumann boundary conditions.
We consider
\[
\int_{\Omega}\pi(t,x) \dx=0 \quad \text{for a.a. } t\in (0,T).
\]
The $L^p$-regularity theory for the Neumann problem \eqref{eq:shortpi} implies, see \cite[Proposition 2.5.2.3]{G}, that
\[
\pi\in L^{p'}(0,T;L^{p'}(\Omega)).
\]
Note that the weak formulation of \eqref{eq:shortpi} is the following identity
\begin{align*}
\int_{\Omega} \pi \Delta \varphi \dx =\int_{\Omega} \S: \nabla^2 \varphi -(\vv \otimes \vv \, g_{k}(|\vv|^2)) :\nabla^2 \varphi -(\T_k(\thet)\mathbf{f} \cdot \nabla \varphi)\dx,
\end{align*}
for all $\varphi\in W^{2,p}(\Omega)$ with $\nabla\varphi\cdot\mathbf{n}=0$ on $\partial\Omega$ and a.a. $t\in(0,T).$ Having such pressure in hands, it is then easy to show that \eqref{eq:finalvdiv} holds as
\begin{equation}\label{eq:finalvpi}
\langle \partial_t\vv,\ww\rangle +\int_{\Omega}(\S-\vv \otimes \vv \, g_{k}(|\vv|^2)) : \nabla \ww - \pi \diver \ww-\T_k(\thet)\mathbf{f}\cdot  \ww\dx=0,
\end{equation}
for all $\ww\in W_{\mathbf{n}}^{1,p}(\Omega)$ and a.a. $t\in(0,T).$
%
%
%
%Using weak lower semicontinuity of norms we  have proved that there exist $(\vv,\thet)$ satisfying
%\begin{align*}
%&\vv\in \mathcal{C}(0,T;L_{\mathbf{n},\diver}^2)\cap L^p(0,T;W_{\mathbf{n},\diver}^{1,p}),\qquad \partial_t\vv\in  L^{p'}(0,T;W_{\mathbf{n},\diver}^{-1,p}),\\
%&\thet\in L^{\infty}(0,T;L^{r_1}(\Omega))\cap L^{\infty}(0,T;W^{1,{r_2}}(\Omega)), \qquad \partial_t\thet \in L^1(0,T;W^{-1,r_2}(\Omega)),\\
%&\log\thet \in L^{\infty}(0,T;L^{1}(\Omega)),
%\end{align*}
%with $p>1,1\leq r_1 <3, 1\leq r_2 <3/2$ and solving \eqref{eq:finalvdiv} and \eqref{eq:renormlimit}.
%
%At this point, the existence of a weak solution $(\vv, \thet) = (\vv^k,\thet^k)$ to the $k$-approximate problem, fulfilling \eqref{eq:finalvpi} and \eqref{eq:renormlimit} is established. What remains in order tor complete the proof of Lemma \eqref{mainlemma} is to establish that the initial conditions are attained.
%
%
%
%\cite{NaPoWo12,NaWo10}

%\bibliographystyle{amsplain}

\providecommand{\bysame}{\leavevmode\hbox to3em{\hrulefill}\thinspace}
\providecommand{\MR}{\relax\ifhmode\unskip\space\fi MR }
\providecommand{\MRhref}[2]{%
  \href{http://www.ams.org/mathscinet-getitem?mr=#1}{#2}
}
\providecommand{\href}[2]{#2}

%\bibliography{bibliography}

\end{document}